\documentclass[a4paper,twoside,11pt]{amsart} %

\usepackage[margin=2.5cm,a4paper]{geometry} %lmargin=2.5cm,rmargin=4cm,tmargin=3cm,bmargin=3cm
\usepackage{pgf,tikz}
\usetikzlibrary{arrows}
\usetikzlibrary{calc}
\usepackage{multirow}
\usepackage{framed}

\usepackage{xpatch}
\makeatletter
\xpatchcmd{\sqrt}{\@sqrt}{{\mspace{-2.9274mu}\@sqrt}}{}{}
\makeatother

\usepackage{caption} 
\usepackage{amsmath, amssymb, amsthm, bm, centernot}
\usepackage{color,booktabs}
\usepackage{url}
\usepackage{enumerate}
\usepackage{hyperref}
\usepackage[capitalise]{cleveref}

\newtheorem{theorem}{Theorem}

\newtheorem{lemma}[theorem]{Lemma}

\numberwithin{theorem}{section}

\theoremstyle{definition}

\theoremstyle{remark}

\theoremstyle{definition}

\newcommand\tile{
\foreach \x in {0,1,2,3,4} {
  \node [vertex] (v\x) at (0,\x) {};
  \node [vertex] (w\x) at (2,\x){};
  \node (ll\x) at (-1.5,\x){};
  \node (l\x) at (-1.15,\x){};
  \node (r\x) at (3.15,\x){};
  \node (rr\x) at (3.5,\x) {};
   \draw [thick, dotted] (ll\x)--(v\x)--(w\x)--(rr\x);
  }
  \draw [thick, dotted] (v0)--(w1)--(v1)--(w2)--(v2)--(w3)--(v3)--(w4)--(v4)--(w0);
}

\newcommand\sixtile{
\foreach \x in {0,1,2,3,4,5} {
  \node [vertex] (v\x) at (0,\x) {};
  \node [vertex] (w\x) at (2,\x){};
  \node (ll\x) at (-1.5,\x){};
  \node (l\x) at (-1.15,\x){};
  \node (r\x) at (3.15,\x){};
  \node (rr\x) at (3.5,\x) {};
   \draw [thick, dotted] (ll\x)--(v\x)--(w\x)--(rr\x);
  }
  \draw [thick, dotted] (v0)--(w1)--(v1)--(w2)--(v2)--(w3)--(v3)--(w4)--(v4)--(w5)--(v5)--(w0);
}

\mathchardef\mh="2D

\swapnumbers

\definecolor{unocc}{rgb}{0.85,0.85,0.85}
\definecolor{occ}{rgb}{1,0,0}

\title{Highly-connected planar cubic graphs with few or many Hamilton cycles}
\author{Irene Pivotto}
\address{Centre for the Mathematics of Symmetry and Computation, Department of Mathematics and Statistics, University of Western Australia}
\email{piv8irene@gmail.com}

\author{Gordon Royle}
\address{Centre for the Mathematics of Symmetry and Computation, Department of Mathematics and Statistics, University of Western Australia}
\email{gordon.royle@uwa.edu.au}

\subjclass[2010]{}


\begin{document}

\begin{abstract}
In this paper we consider the number of Hamilton cycles in planar cubic graphs of high cyclic edge-connectivity, answering two questions
raised by Chia and Thomassen (``On the number of longest and almost longest cycles in cubic graphs'', Ars Combin., 104, 307--320, 2012) about extremal graphs in these families. In particular, we find families of cyclically $5$-edge connected planar cubic graphs with more Hamilton cycles than the generalized Petersen graphs $P(2n,2)$. The graphs themselves are fullerene graphs that correspond to certain carbon molecules known as nanotubes --- more precisely, the family consists of the zigzag nanotubes of (fixed) width $5$ and increasing length. In order to count the Hamilton cycles in the nanotubes, we develop methods inspired by the transfer matrices of statistical physics. We outline how these methods can be adapted to count the Hamilton cycles in nanotubes of greater (but still fixed) width, with the caveat that the resulting expressions involve matrix powers. We also consider cyclically $4$-edge-connected cubic planar graphs with few Hamilton cycles, and exhibit an infinite family of such graphs each with exactly $4$ Hamilton cycles. Finally we consider the ``other extreme'' for these two classes of graphs, thus investigating cyclically $4$-edge connected cubic planar graphs with \emph{many} Hamilton cycles and the cyclically $5$-edge connected cubic planar graphs with \emph{few} Hamilton cycles. In each of these cases, we present partial results, examples and conjectures regarding the graphs with few or many Hamilton cycles.
\end{abstract}

%In this section we consider the five problems raised by Chia \& Thomassen on the numbers of Hamilton cycles (and longest cycles) in various classes of cubic graphs.  Although some of these questions (which we have reproduced as they occur in \cite{}) refer to ``longest cycles'' in order to encompass both Hamiltonian and non-Hamiltonian graphs, we will only consider the
%questions as they apply to Hamiltonian graphs.

\maketitle

\section{Introduction}

The study of Hamilton cycles in cubic graphs has an extensive history, initially driven by attempts to prove Tait's conjecture that every $3$-connected planar cubic graph is Hamiltonian. If the conjecture were true, the Four-Colour Theorem would have been an immediate corollary, but this possible avenue to a proof was firmly closed nearly $60$ years later by Tutte's \cite{MR0019300} discovery of a  non-Hamiltonian $3$-connected planar cubic graph. In this same paper, Tutte also presented an ingenious parity argument that a cubic graph cannot have a \emph{unique} Hamilton cycle, crediting the result (but not this particular proof) to his friend and colleague C.A.B Smith. (Unfortunately, Smith's original proof appears to have been lost.) As a consequence, a Hamiltonian cubic graph has at least three Hamilton cycles, and while the \emph{planar} cubic graphs with exactly three Hamilton cycles are known (Fowler \cite{MR2698714}), a full classification is not. 

Many authors have considered the \emph{existence} of Hamilton cycles for restricted classes of cubic graphs defined by imposing additional conditions such as planarity, bipartiteness, high connectivity and/or combinations of such conditions, often leading to difficult still-unsolved problems. The most prominent of these is probably Barnette's conjecture that a three-connected cubic planar \emph{bipartite} graph is Hamiltonian.   There is a smaller literature concerning the \emph{enumeration} of Hamilton cycles, though there are some notable results such as Schwenk's \cite{MR1007713} enumeration of the Hamilton cycles in the generalized Petersen graphs $P(m,2)$ (which are defined below). It is clear that exact enumeration results will only ever be possible for explicitly-defined graphs or families of graphs, and only then if the graphs are highly structured.  Where exact results are impossible, researchers usually try to determine \emph{bounds} on the values under consideration. Along these lines, Chia and Thomassen \cite{MR2951794} considered a number of classes of cubic graphs defined according to various connectivity and planarity constraints and investigated upper- and lower-bounds (as functions of the number of vertices) for the number of Hamilton cycles of a graph in each class.
In fact, they actually considered the slightly broader question of bounding the numbers of \emph{longest cycles}, in order to include possibly-non-Hamiltonian graphs in their results. However in this paper, we restrict our attention to Hamiltonian graphs only giving us the main question---If a graph in a particular class has at least one Hamilton cycle, then how many more \emph{must} it have and how many more \emph{can} it have.

Accompanying their results and examples, Chia and Thomassen posed five open problems regarding the number of Hamilton cycles in different classes of cubic graphs. In this paper, we answer two of their questions (which happen to be their Question 1 and Question 5), each of which relates to a class of planar cubic graphs of a particular minimum cyclic edge-connectivity. 
\begin{itemize}
\item [{\bf Q1}] Does every planar cubic cyclically $4$-edge-connected graph on $n$ vertices contain at least $n/2$ longest cycles?
\item [{\bf Q5}] Does there exist a planar cubic cyclically $5$-edge-connected graph on $4n$ vertices with more longest cycles than $P(2n,2)$? 
\end{itemize}
We recall here that 
a graph is \emph{cyclically $k$-edge-connected} if there is no edge-set with fewer than $k$ edges whose deletion leaves at least two connected components each containing a cycle. The constraint that each connected component contain a cycle is designed to rule out trivial edge cuts, such as the edges incident with a single vertex or the four edges incident with exactly one of the endpoints of a  fixed edge. So while a cubic graph is at most $3$-edge-connected under the normal definition, its cyclic edge-connectivity can be arbitrarily large. One family of graphs that crops up frequently in any discussion of Hamilton cycles in cubic graphs is the family of generalized Petersen graphs, defined as follows: $P(m,k)$ has $2m$ vertices $\{u_0,u_1,\ldots,u_{m-1}\} \cup \{v_0,v_1,\ldots,v_{m-1}\}$ where for each $i$,  vertex $u_i$ is adjacent to $u_{i+1}$,  vertex $v_i$ is adjacent to $v_{i+k}$, and $u_i$ is adjacent to $v_i$, with all indices being taken modulo $m$.  Generalized Petersen graphs crop up frequently in the study of cubic graphs, for example $P(5,2)$ is the Petersen graph and $P(10,2)$ is the dodecahedron.

%Although our terminology and notation is standard and should cause no confusion, for completeness we highlight a few terms. Our graphs are
%simple, having no loops and multiple edges, and a cubic (or $3$-regular) graph is one where each vertex has exactly $3$ neighbours. A Hamilton cycle $C$ in $G$ is a set of edges such that the graph $(V(G), E(C))$ is a $2$-regular graph with a single component. 

The main results of this paper are that the answers to these two questions are ``No" and ``Yes'' respectively. In \cref{cub4confew} we answer Problem 1 by showing that there is an infinite family of planar cubic cyclically $4$-edge-connected graphs with a \emph{constant} number of Hamilton cycles, independent of the number of vertices in the graph.  We note that this result was recently independently obtained by Goedgebeur, Meersman and Zamfirescu \cite{1812.05650}, along with a number of additional attractive results regarding not-necessarily-cubic Hamiltonian graphs with few Hamilton cycles.

In \cref{cub5conmany} we answer Problem 5 by showing that a certain family 
of \emph{fullerene graphs} (planar cubic graphs with faces of size $5$ and $6$ only)
known to mathematical chemists as ``\emph{capped $(5,0)$ zigzag nanotubes"} has more Hamilton cycles than the generalized Petersen graph $P(2n,2)$ whenever the number of vertices is a multiple of $20$ greater than $20$.

In order to count the numbers of Hamilton cycles in the nanotubes and related graphs, we use methods inspired by the transfer matrix methods of statistical physics. In addition to counting the number of Hamilton cycles in our particular class of nanotubes, we give a general procedure that could -- in principle -- be carried out for nanotubes of any fixed moderate width. For $n \in \{5,6\}$, the graphs $P(2n,2)$ are themselves fullerenes, but as they have two faces of length $n$, they look less and less fullerene-like as $n$ increases. It is therefore somewhat surprising that the graphs with more Hamilton cycles than $P(2n,2)$ are themselves fullerenes.  However it turns out that nanotubes are not the \emph{only} fullerenes that have more Hamilton cycles than the generalized Petersen graphs, as our computations revealed an additional example on $56$ vertices. We suspect that this graph lies in another infinite family of fullerenes particularly rich in Hamilton cycles, but we have not yet been able to describe this family in general.

%
%and
%in fact, computations reveal that there is another fullerene on $56$ vertices that beats the generalized Petersen graphs. Further computations on 
%larger numbers of vertices 
%
%
%We investigated this
%phenomenon computationally to see if there are \emph{other} fullerenes also with more Hamilton cycles than $P(2n,2)$. We examined all the
%fullerenes on up to $100$ vertices, and discovered that in addition to the nanotubes described above, there are two graphs, on $56$ and $96$ vertices,
%that have a greater number of Hamilton cycles than the generalized Petersen graphs on the same numbers of vertices. As fullerenes, these graphs look
%qualitatively similar in that their $12$ pentagonal faces are arranged in three connected patches with each patch containing four pentagons and being 
%as widely separated from the other patches as possible.

Computationally, it is relatively straightforward to construct planar cubic cyclically $k$-edge-connected graphs using Brinkmann \& McKay's program {\tt plantri} \cite{MR2357364} and then count their Hamilton cycles. (The numbers of such graphs can be found in Sloane's Encyclopaedia of Integer Sequences (\cite{OEIS}) at \url{http://oeis.org/A007021} for $k=4$ and \url{http://oeis.org/A006791} for $k=5$.) We performed these computations for all the graphs on up to $42$ vertices for $k=4$, and up to $64$ vertices for $k=5$. The results of these computations are given in Table~\ref{tab:plancubcyc4con} and Table~\ref{tab:plancubcyc5con}. Note that while the total numbers of graphs includes graphs that are non-Hamiltonian, the column ``Min'' lists the minimum number of Hamilton cycles in a Hamiltonian graph in that family.
(We note that Goedgebeur, Meersman and Zamfirescu \cite{1812.05650} have more extensive computations for this class of graphs in their investigation of graphs with few Hamilton cycles.)

%\begin{table}
%\begin{tabular}{|c|c|c|}
%\hline
% & cyclically 4-edge-connected & cyclically 5-edge-connected \\\hline 
%
% \multirow{ 4}{*}{ planar} & rings of ladders & $5 \cdot 2^{k} + 20 \cdot 12^{(k-1)/2}$ ($n=10+10k$)\\ 
% & new family & nanotubes + 56 and 96 vert fullerenes\\ \cline{2-3}
%& $4$ & $16$\\ 
% & ladder extensions & only 10 examples (starting at 52 vert)\\ \hline
%
% \multirow{ 4}{*}{ no restriction} & $ \frac{3}{2}\cdot 2^{\frac{n}{4}}$ & $4^{\frac{n}{4}}+2^{\frac{n}{4}+1}$\\ 
% & ring of 4-cycles + other messy ones & Catherine wheels\\ \cline{2-3}
%& $3$ & $3$\\ 
% & $P(6k+3,2)$  & $P(6k+3,2)$\\ \hline
%\end{tabular}
%\end{table}

\begin{table}
\begin{center}
\begin{tabular}{rrrr}
\toprule
$n$ & No. graphs & Min & Max \\
\midrule
$10$ & $1$ & $5$ &$5$ \\
$12$ & $2$ & $7$ & $8$ \\
$14$ & $4$ & $6$ & $9$\\
$16$ & $10$ & $6$ & $14$\\
$18$ & $25$ & $8$ & $20$\\
\midrule
$20$ & $87$ & $9$ & $30$\\
$22$ & $313$ & $10$ & $32$\\
$24$ & $1357$ & $11$ & $48$\\
$26$ & $6244$ & $10$ & $56$\\
$28$ & $30926$ & $12$ & $80$\\
\midrule
$30$ & $158428$ & $10$ & $112$\\
$32$ & $836749$ & $12$ & $148$\\
$34$ & $4504607$ & $8$ & $199$ \\
$36$ & $24649284$ & $8$ & $276$ \\
$38$ & $136610879$ & $4$ & $368$\\
\midrule
$40$ & $765598927$ & $8$ & $542$\\
$42$ & $4332047595$ &  $4$ & $717$ \\
\bottomrule
\end{tabular}
\end{center}
\caption{Hamilton cycles in planar cubic cyclically $4$-edge-connected graphs}
\label{tab:plancubcyc4con}
\end{table}

\begin{table}
\begin{tabular}{crrrc}
\toprule
$n$ & No. graphs& Min & Max & Most cycles\\
\midrule
$20$& $1$ & $30$ & $30$&$P(10,2)$\\
$22$& $0$\\
$24$&$1$&$34$&$34$&$P(12,2)$\\
$26$&$1$&$24$&$24$\\
$28$&$3$&$18$&$56$&$P(14,2)$\\
$30$&$4$&$20$&$52$\\
\midrule
$32$&$12$&$30$&$108$&$P(16,2)$ \\
$34$&$23 $&$28$&$100$ \\
$36$& $71$&$24$&$150$&$P(18,2)$\\
$38$&$187$&$24$&$168$ \\
$40$&$627$ &$32$&$280$&$N(5,3)$\\
\midrule
$42$&$1970$&$28$&$244$ \\
$44$&$6833$ &$28$&$418$&$P(22,2)$\\
$46$&$23384$&$36$&$390$ \\
$48$&$82625$&$32$&$642$&$P(24,2)$ \\
$50$&$292164$&$32$&$780$ \\
\midrule
$52$&$1045329$&$16$&$1040$ & $P(26,2)$ \\
$54$&$3750277$ &$16$&$1120$\\
$56$&$13532724$ & $16$ &$1746$ & ? \\
$58$&$ 48977625$ & $16$ & $1928$ &  \\
$60$&$177919099$& $16$ & $3040$ & $N(5,5)$\\
\midrule
$62$&$648145255$ & $16$ & $3540$ & \\
$64$&$2368046117$ & $16$ & $4412$ & $P(32,2)$\\
\bottomrule
\end{tabular}
\caption{Hamilton cycles in planar cubic cyclically $5$-edge-connected graphs}
\label{tab:plancubcyc5con}
\end{table}

Given that we have already computed the necessary data, we also consider the ``other extreme'' for each of these two classes of planar cubic graphs. In other words, we consider the cyclically $4$-edge connected planar cubic graphs with the \emph{most} Hamilton cycles in \cref{cub4conmany}, and the cyclically $5$-edge connected planar cubic graphs with the \emph{fewest} Hamilton cycles in \cref{cub5confew}.

\section{Planar cubic cyclically \texorpdfstring{$4$}{4}-edge-connected graphs with few Hamilton cycles} 
\label{cub4confew}

The first class of graphs that we consider is the class of planar cubic cyclically $4$-edge-connected graphs, focussing first on those with few Hamilton cycles thereby addressing the first question from Chia and Thomassen:  \emph{``Does every planar cubic cyclically $4$-edge-connected graph on $n$ vertices have at least $n/2$ longest cycles?''}. We note that this question is essentially asking whether there is a lower bound on the number of Hamilton cycles that is \emph{linear} in the number of vertices. Table~\ref{tab:plancubcyc4con} shows that this is not true for most of the values of $n$ within our computational range, and in this section we show that this continues indefinitely by constructing an infinite family of planar cubic cyclically $4$-edge-connected graphs, each with a \emph{constant} number of Hamilton cycles. 
More precisely, we exhibit a $38$-vertex planar cubic cyclically $4$-edge-connected graph with exactly four Hamilton cycles that acts as a ``starter'' graph. We then show that this graph contains special configurations that can be extended by the addition of any multiple of $4$ vertices in such a way that each graph constructed is cubic, planar, cyclically $4$-edge-connected, and has the same number of Hamilton cycles as the starter graph. 

As noted previously, this has been independently proved by Goedgebeur, Meersman and Zamfirescu \cite{1812.05650}.

The construction adding four vertices proceeds as follows: 
Let $C$ be an induced 4-cycle in a graph $G$, with edges $v_1v_2$, $v_2v_3$, $v_3v_4$, and $v_4v_1$. Construct a graph $H$ from $G$ by 
replacing the edge $v_1v_2$ with a path $v_1x_1x_2v_2$ and replacing the edge $v_3v_4$ with a path $v_3y_2y_1v_4$, where $x_1,x_2,y_1,y_2$ are new vertices. Finally, add edges $x_1y_1$ and $x_2y_2$. Then we say that $H$ is obtained from $G$ by a \emph{ladder extension} on edges $v_1v_2$ and $v_3v_4$; we are extending a ``$2$-rung ladder'' to a $4$-rung ladder.
The next lemma shows that --- under certain circumstances --- performing a ladder extension does not alter the number of Hamilton cycles in the graph.

\begin{lemma}
Suppose that $C$ is an induced $4$-cycle in a graph $G$ with vertices $\{v_1,v_2,v_3,v_4\}$ (as in Figure~\ref{fig:ladext}), that all the vertices of $C$ have degree 3 in $G$, and that $G \mathbin{\backslash} \{v_1v_2,v_3v_4\}$ is not Hamiltonian. Furthermore, let $H$ be the graph obtained from $G$ by a ladder extension on $v_1v_2$ and $v_3v_4$. Then there is a bijection between the Hamilton cycles of $G$ and the Hamilton cycles of $H$. 
\end{lemma}

\begin{proof}
We start by considering a Hamilton cycle $D$ of $G$. Then either $D$ uses both $v_1v_2$ and $v_3v_4$ or, without loss of generality, $E(D)\cap E(C)=\{v_4v_1,v_1v_2, v_2v_3\}$. In the first case,  replacing the edge $v_1v_2$ with the path $v_1x_1x_2v_2$ and the edge $v_3v_4$ with the path $v_3 y_2 y_1 v_4$ yields a Hamilton cycle of $H$. Moreover, all Hamilton cycles of $H$ using these two paths arise from Hamilton cycles of $G$ in this way. In the second case, replacing the edge $v_1v_2$ with the path $v_1x_1y_1y_2x_2v_2$ yields a Hamilton cycle of $H$. Moreover, all Hamilton cycles of $H$ using this path arise from Hamilton cycles of $G$ in this way.

Up to symmetry, any further Hamilton cycle in $H$ would necessarily use the paths $v_1x_1y_1v_4$ and $v_2x_2y_2v_3$ or use the path $v_1x_1x_2y_2y_1v_4$ and edge $v_2v_3$. In either of these two cases, this Hamilton cycle would correspond to a Hamilton cycle of $G$ using only the edges $v_1v_4$ and $v_2v_3$ of $C$, a contradiction to the condition that $G \mathbin{\backslash} \{v_1v_2,v_3v_4\}$ is not Hamiltonian.   \end{proof}

We call a $4$-cycle \emph{extendable} if it meets the hypotheses of the lemma, and note that after performing a ladder extension, the newly-created $4$-cycles in $H[v_1,v_2,v_3,v_4,x_1,x_2,y_1,y_2]$ are also extendable. Therefore the process can be repeated, extending a $4$-cycle to an arbitrarily long ladder (provided the total number of rungs is even), thus creating an infinite family of graphs with the same number of Hamilton cycles as the original graph.

\begin{figure}
\begin{center}
\begin{tikzpicture}[scale=1.5]
\tikzstyle{vertex}=[circle, draw=black, fill=black!15!white, inner sep = 0.65mm]
\node [vertex,label={270:{\small $v_1$}}] (v1) at (0,0) {};
\node [vertex,label={270:{\small $v_2$}}] (v2) at (1,0) {};
\node [vertex,label={90:{\small $v_3$}}] (v3) at (1,1) {};
\node [vertex,label={90:{\small $v_4$}}] (v4) at (0,1) {};
\draw (v1)--(v2)--(v3)--(v4)--(v1);
\draw [vertex] (v1) -- ++(-0.3,0);
\draw [vertex] (v2) -- ++(0.3,0);
\draw [vertex] (v3) -- ++(0.3,0);
\draw [vertex] (v4) -- ++(-0.3,0);
\node at (-0.5,0.5) {$G$};

\draw [thick, ->] (1.75,0.5)--(2.25,0.5);
\pgftransformxshift{3cm}
\node [vertex,label={270:{\small $v_1$}}](v1) at (0,0) {};
\node [vertex,label={270:{\small $v_2$}}] (v2) at (1.5,0) {};
\node [vertex,label={90:{\small $v_3$}}] (v3) at (1.5,1) {};
\node [vertex,label={90:{\small $v_4$}}] (v4) at (0,1) {};

\node [vertex,label={270:{\small $x_1$}}](x1) at (0.5,0) {};
\node [vertex,label={270:{\small $x_2$}}](x2) at (1,0) {};
\node [vertex,label={90:{\small $y_1$}}](y1) at (0.5,1) {};
\node [vertex,label={90:{\small $y_2$}}](y2) at (1,1) {};
\draw (v1)--(x1)--(x2)--(v2)--(v3)--(y2)--(y1)--(v4)--(v1);
\draw [vertex] (v1) -- ++(-0.3,0);
\draw [vertex] (v2) -- ++(0.3,0);
\draw [vertex] (v3) -- ++(0.3,0);
\draw [vertex] (v4) -- ++(-0.3,0);
\draw (x1)--(y1);
\draw (x2)--(y2);
\node at (2,0.5) {$H$};
\end{tikzpicture}
\end{center}
\caption{Ladder extension on an induced $4$-cycle}
\label{fig:ladext}
\end{figure}
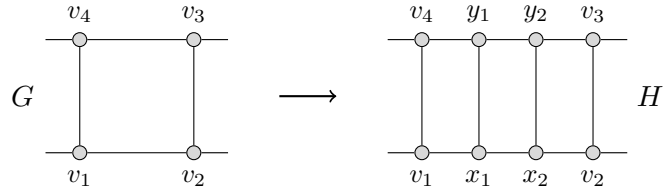

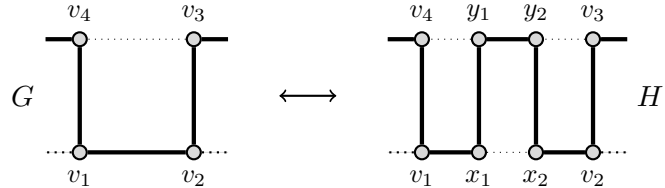
\begin{figure}
\begin{center}
\begin{tikzpicture}[scale=1.5]
\tikzstyle{vertex}=[circle, draw=black, thick, fill=black!15!white, inner sep = 0.65mm]
\node [vertex,label={270:{\small $v_1$}}] (v1) at (0,0) {};
\node [vertex,label={270:{\small $v_2$}}] (v2) at (1,0) {};
\node [vertex,label={90:{\small $v_3$}}] (v3) at (1,1) {};
\node [vertex,label={90:{\small $v_4$}}] (v4) at (0,1) {};
\draw  [dotted]  (v1)--(v2)--(v3)--(v4)--(v1);
\draw [ dotted,vertex] (v1) -- ++(-0.3,0);
\draw [dotted,vertex] (v2) -- ++(0.3,0);
\draw [dotted,vertex] (v3) -- ++(0.3,0);
\draw [dotted,vertex] (v4) -- ++(-0.3,0);
\node at (-0.5,0.5) {$G$};

\draw [thick, <->] (1.75,0.5)--(2.25,0.5);

\draw [ultra thick] (v4)--(v1)--(v2)--(v3);
\draw [ultra thick] (v4) -- ++(-0.3,0);
\draw [ultra thick] (v3) -- ++(0.3,0);

\pgftransformxshift{3cm}
\node [vertex,label={270:{\small $v_1$}}](v1) at (0,0) {};
\node [vertex,label={270:{\small $v_2$}}] (v2) at (1.5,0) {};
\node [vertex,label={90:{\small $v_3$}}] (v3) at (1.5,1) {};
\node [vertex,label={90:{\small $v_4$}}] (v4) at (0,1) {};

\node [vertex,label={270:{\small $x_1$}}](x1) at (0.5,0) {};
\node [vertex,label={270:{\small $x_2$}}](x2) at (1,0) {};
\node [vertex,label={90:{\small $y_1$}}](y1) at (0.5,1) {};
\node [vertex,label={90:{\small $y_2$}}](y2) at (1,1) {};
\draw [dotted] (v1)--(x1)--(x2)--(v2)--(v3)--(y2)--(y1)--(v4)--(v1);
\draw [dotted,vertex] (v1) -- ++(-0.3,0);
\draw [dotted,vertex] (v2) -- ++(0.3,0);
\draw [dotted,vertex] (v3) -- ++(0.3,0);
\draw [dotted,vertex] (v4) -- ++(-0.3,0);
\draw [dotted] (x1)--(y1);
\draw [dotted] (x2)--(y2);

\draw [ultra thick] (v4)--(v1)--(x1)--(y1)--(y2)--(x2)--(v2)--(v3);
\draw [ultra thick] (v4) -- ++(-0.3,0);
\draw [ultra thick] (v3) -- ++(0.3,0);

\node at (2,0.5) {$H$};
\end{tikzpicture}
\end{center}
\caption{Corresponding Hamilton cycles in $G$ and $H$}
\label{fig:ladrep}
\end{figure}

\begin{theorem}
For all $n \geqslant 38$ such that $n \equiv 2 \pmod{4}$, there is a planar cubic cyclically $4$-edge-connected graph on $n$
vertices with exactly four Hamilton cycles. 
\end{theorem}

\begin{proof}
It can be verified, most easily by computer, that the graph of Figure~\ref{basegraph38} has  exactly four Hamilton cycles.  In fact, as illustrated in the second figure, $30$ edges lie in all four Hamilton cycles and $11$ lie in no Hamilton cycles, leaving just $8$ out of the remaining $16$ edges to (potentially) complete a Hamilton cycle. These ``undecided'' edges (drawn dashed in the figure) fall into two $8$-cycles, and it is easy to see that a Hamilton cycle must use exactly four pairwise disjoint edges (i.e. a perfect matching) from each of these $8$-cycles. As each cycle has exactly two perfect matchings, there are only four choices for the additional $8$ edges, 
and all four of the possible choices yield a Hamilton cycle. 

The graph has two $4$-cycles and as every Hamilton cycle uses the two (roughly) vertical edges in each of those, 
both of the $4$-cycles are extendable, and so can be extended to ladders of arbitrary even length.
\end{proof}

%As is well known, a Hamiltonian cubic graph cannot have a unique Hamilton cycle, nor two Hamilton cycles.
%It is a long-standing conjecture that a  planar cubic graph with exactly $3$ Hamilton cycles must
%contain a triangle. (It is known (Fowler \cite{}) that uniquely $3$-edge-colourable planar cubic graphs must contain
%a triangle, but it is still \emph{possible} that there is a  planar cubic graph with exactly $3$ Hamilton cycles that
%is not uniquely $3$-edge colourable.)

As the graph of Figure~\ref{basegraph38} has two extendable $4$-cycles and each can independently be extended to a ladder with an arbitrary (even) number of rungs, this gives to a slowly increasing number of examples as the number of vertices increases. We expect that there will be another graph on $n \equiv 0 \pmod{4}$ vertices with exactly four Hamilton cycles and extendable cycles. 

Goedgebeur, Meersman and Zamfirescu \cite{1812.05650} report that there are planar cubic cyclically $4$-edge connected graphs on $48$ vertices with exactly four Hamilton cycles.

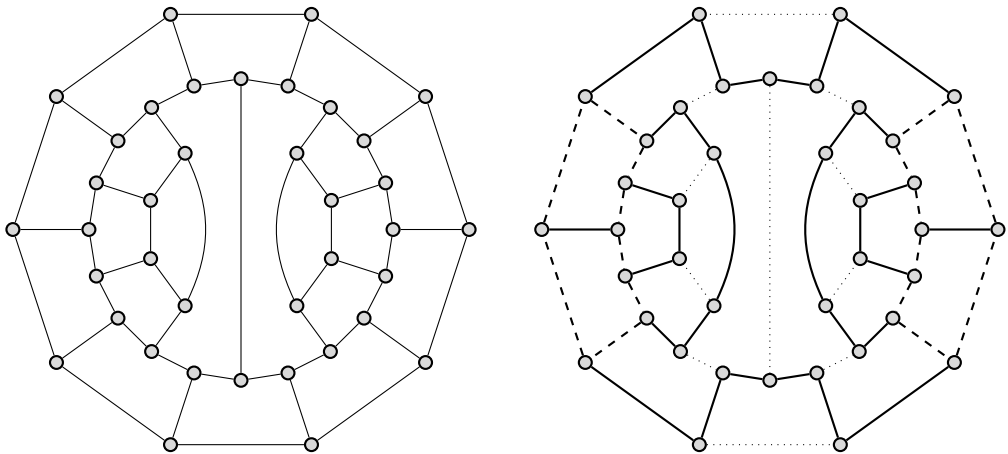
\begin{figure}
\begin{center}
\begin{tikzpicture}
\tikzstyle{vertex}=[circle, draw=black, thick, fill=black!15!white, inner sep = 0.6mm]
%\draw (0,0) circle (3cm);
%\draw (0,0) circle (2cm);
\foreach \x in {0,1,...,19} {
\node [vertex] (a\x) at (18*\x:2cm) {};
}
\foreach \x in {0,2,4,6,8,10,12,14,16,18} {
 \node [vertex] (b\x) at (18*\x:3cm) {};
 \draw (a\x)--(b\x);
}
\foreach \x in {1,3,7,9,11,13,17,19} {
\node[vertex] (c\x) at (18*\x:1.25cm) {};
\draw (a\x)--(c\x);
}

\draw (a0)--(a1)--(a2)--(a3)--(a4)--(a5)--(a6)--(a7)--(a8)--(a9)--(a10)--(a11)--(a12)--(a13)--(a14)--(a15)--(a16)--(a17)--(a18)--(a19)--(a0);

\draw (b0)--(b2)--(b4)--(b6)--(b8)--(b10)--(b12)--(b14)--(b16)--(b18)--(b0);

\draw (a5)--(a15);

\draw (c7)--(c9)--(c11)--(c13);
\draw (c17)--(c19)--(c1)--(c3);
\draw [bend left = 25] (c7) to (c13);
\draw [bend left = 25] (c17) to (c3);
 
 \end{tikzpicture}
 \hspace{0.5cm}
 \begin{tikzpicture}
\tikzstyle{vertex}=[circle, draw=black, thick, fill=black!15!white, inner sep = 0.6mm]
%\draw (0,0) circle (3cm);
%\draw (0,0) circle (2cm);
\foreach \x in {0,1,...,19} {
\node [vertex] (a\x) at (18*\x:2cm) {};
}
\foreach \x in {0,2,4,6,8,10,12,14,16,18} {
 \node [vertex] (b\x) at (18*\x:3cm) {};
% \draw (a\x)--(b\x);
}
\foreach \x in {1,3,7,9,11,13,17,19} {
\node[vertex] (c\x) at (18*\x:1.25cm) {};
%\draw (a\x)--(c\x);
}

%\draw (a5)--(a15);
%
%\draw (c7)--(c9)--(c11)--(c13);
%\draw (c17)--(c19)--(c1)--(c3);
%\draw [bend left = 25] (c7) to (c13);
%\draw [bend left = 25] (c17) to (c3);

\draw [thick, black] (a8)--(a7)--(c7);
\draw [thick, black, bend left = 25] (c7) to (c13);
\draw [thick, black] (c13)--(a13)--(a12);
\draw [thick, black] (a9)--(c9)--(c11)--(a11);
\draw [thick, black] (a10)--(b10);

\draw [thick, black] (a2)--(a3)--(c3);
\draw [thick, black, bend left = 25] (c17) to (c3);
\draw [thick, black] (c17)--(a17)--(a18);
\draw [thick, black] (a1)--(c1)--(c19)--(a19);
\draw [thick, black] (a0)--(b0);

\draw [thick, black] (b2)--(b4)--(a4)--(a5)--(a6)--(b6)--(b8);
\draw [thick, black] (b12)--(b14)--(a14)--(a15)--(a16)--(b16)--(b18);

\draw [dotted] (b4)--(b6);
\draw [dotted] (a3)--(a4);
\draw [dotted] (a6)--(a7);
\draw [dotted] (c7)--(c9);
\draw [dotted] (b14)--(b16);
\draw [dotted] (c1)--(c3);
\draw [dotted] (a16)--(a17);
\draw [dotted] (a13)--(a14);
\draw [dotted] (c13)--(c11);
\draw [dotted] (a5)--(a15);
\draw [dotted] (c17)--(c19);

\draw [thick, dashed] (b2)--(a2)--(a1)--(a0)--(a19)--(a18)--(b18)--(b0)--(b2);
\draw [thick, dashed] (b8)--(a8)--(a9)--(a10)--(a11)--(a12)--(b12)--(b10)--(b8);
 
 \end{tikzpicture}
\end{center}
\caption{On the left, a cyclically $4$-edge-connected cubic planar graph with four Hamilton cycles. 
On the right, edges are coded black, dashed and dotted according as they lie in all, half or none of
the Hamilton cycles respectively.}
\label{basegraph38}
\end{figure}

%Here \tikz \draw[thick, black] (0,0) -- (1,0); is the 

\section{Planar cubic cyclically \texorpdfstring{$4$}{4}-edge-connected graphs with many Hamilton cycles}
\label{cub4conmany}

In this section, we consider planar cubic cyclically $4$-edge connected graphs with \emph{many} Hamilton cycles. The computational evidence reveals definite patterns in the extremal graphs in this class, although probably too complicated for an exact characterisation. We content ourselves with identifying a family of graphs with many Hamilton cycles that we believe to be extremal infinitely often, when $n$ satisfies certain parity conditions.

A \emph{$k$-rung ladder} (or just \emph{$k$-ladder} for short) is the Cartesian product $P_k \mathbin{\square} K_2$ of a $k$-vertex path $P_k$ with $K_2$, which we will assume is  labelled as in \cref{fig:4rungladder}. Borrowing terminology from the physical ladders found in hardware stores, the two copies of $P_k$ are called the \emph{rails} of the ladder (so each ladder has a $u$-rail and a $v$-rail) and the \emph{rungs} are the edges connecting the rails.

\begin{figure}
\begin{center}
\begin{tikzpicture}
\tikzstyle{vertex}=[circle, draw=black, thick, fill=black!15!white, inner sep = 0.6mm]
\foreach \x in {1,2,...,4 }{
  \node [vertex, label=0:$v_\x$] (v\x) at (1,3-0.75*\x) {};
  \node [vertex, label=180:$u_\x$] (u\x) at (0,3-0.75*\x) {};
  \draw (u\x)--(v\x);
}
\draw (u1)--(u2)--(u3)--(u4);
\draw (v1)--(v2)--(v3)--(v4);

\pgftransformxshift{6cm}
\pgftransformyshift{1.5cm}

\foreach \x in {0,1,...,9} {
  \foreach \y in {0,1,2,3} {
    \node [vertex] (v\x\y) at (36*\x:1+0.5*\y) {};
  }
}
\draw (v00)--(v10)--(v20)--(v30)--(v40)--(v50)--(v60)--(v70)--(v80)--(v90)--(v00);
\draw (v03)--(v13)--(v23)--(v33)--(v43)--(v53)--(v63)--(v73)--(v83)--(v93)--(v03);

\foreach \x in {0,1,...,9} {
  \draw (v\x0)--(v\x1)--(v\x2)--(v\x3);
}

\foreach \x/\y in {2/3, 4/5, 6/7, 8/9, 0/1} {
\draw (v\x1)--(v\y1);
\draw (v\x2)--(v\y2);
}

\foreach \x in {0,1,...,4} {
\node at (90+72*\x:2.7cm) {$L_\x$};
}
\end{tikzpicture}
\end{center}
\caption{A $4$-rung ladder and a $5$-ring of $4$-ladders}
\label{fig:4rungladder}
\end{figure}
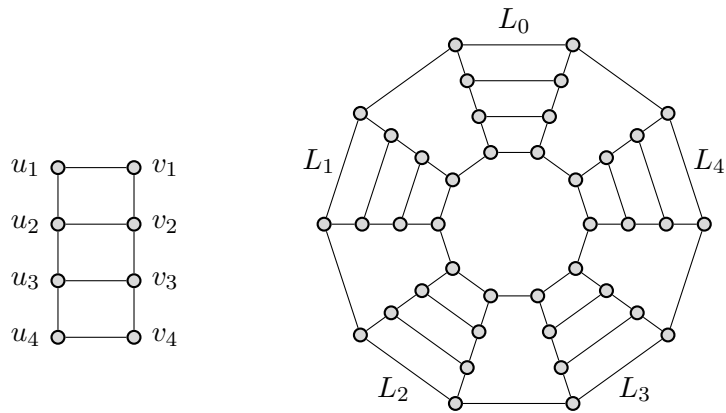

For $m$, $k\geqslant 2$, we define an \emph{$m$-ring of $k$-ladders}, denoted $\mathrm{RL}(m,k)$, to be the graph obtained by connecting $m$ disjoint $k$-ladders $L_0$, $L_1$, $\ldots$, $L_{m-1}$ in a cycle. Adjacent ladders in the circular order are joined by adding edges connecting the vertices $u_1$ and $u_k$ in $L_i$ to the vertices $v_1$ and $v_k$ (respectively) in $L_{i+1}$. (This construction is most easily understood by glancing at Figure~\ref{fig:4rungladder}.)  The graph $\mathrm{RL}(m,k)$ is a planar cubic cyclically $4$-edge connected graph on $2mk$ vertices.

%$H_1,\ldots,H_m$ of a $k$-ladder segment by
%joining each $H_i$ with $H_{i+1}$ with two edges in a circular fashion, as in Figure~\ref{fig:ringladder}.
%\begin{figure}[htbp]
%\begin{center}
%\includegraphics[width=6cm]{figures/ringladder.pdf}
%\caption{A $5$-ring of $4$-ladders.}
%\label{fig:ringladder}
%\end{center}
%\end{figure}
%In an $m$-ring of $k$-ladders there are  $n=2mk$ vertices and the numberof Hamilton cycles depends on the parity of $k$.

%\begin{figure}
%\begin{center}
%\begin{tikzpicture}
%\tikzstyle{vertex}=[circle, draw=black, fill=black!25!white, inner sep = 0.6mm]
%\foreach \x in {0,1,...,9} {
%  \foreach \y in {0,1,2,3} {
%    \node [vertex] (v\x\y) at (36*\x:1+0.5*\y) {};
%  }
%}
%\draw (v00)--(v10)--(v20)--(v30)--(v40)--(v50)--(v60)--(v70)--(v80)--(v90)--(v00);
%\draw (v03)--(v13)--(v23)--(v33)--(v43)--(v53)--(v63)--(v73)--(v83)--(v93)--(v03);
%
%\foreach \x in {0,1,...,9} {
%  \draw (v\x0)--(v\x1)--(v\x2)--(v\x3);
%}
%
%\foreach \x/\y in {2/3, 4/5, 6/7, 8/9, 0/1} {
%\draw (v\x1)--(v\y1);
%\draw (v\x2)--(v\y2);
%}
%
%\end{tikzpicture}
%\end{center}
%\caption{$\text{RL}(5,4)$ is a $5$-ring of $4$-ladders}
%\label{fig:rl54}
%
%\end{figure}

%\begin{figure}
%\begin{center}
%\begin{tikzpicture}
%\tikzstyle{vertex}=[circle, draw=black, fill=black!25!white, inner sep = 0.6mm]
%
%
%
%
%\end{tikzpicture}
%\end{center}
%\end{figure}

\begin{figure}
\begin{center}
\begin{tikzpicture}[scale=0.8] 
\tikzstyle{vertex}=[circle, draw=black, thick, fill=black!15!white, inner sep = 0.6mm]
\foreach \x in {0,1,...,9} {
  \foreach \y in {0,1,2,3} {
    \node [vertex] (v\x\y) at (36*\x:1+0.6*\y) {};
  }
}
\draw [thick, dotted] (v00)--(v10)--(v20)--(v30)--(v40)--(v50)--(v60)--(v70)--(v80)--(v90)--(v00);
\draw [thick, dotted] (v03)--(v13)--(v23)--(v33)--(v43)--(v53)--(v63)--(v73)--(v83)--(v93)--(v03);

%\draw [ultra thick] (v00)--(v

\foreach \x in {0,1,...,9} {
  \draw [thick, dotted] (v\x0)--(v\x1)--(v\x2)--(v\x3);
}

\foreach \x/\y in {2/3, 4/5, 6/7, 8/9, 0/1} {
\draw [thick, dotted] (v\x1)--(v\y1);
\draw [thick, dotted] (v\x2)--(v\y2);
}

\draw [ultra thick] (v00)--(v01)--(v02)--(v12)--(v11)--(v10)--(v20)--(v21)--(v31)--(v30)--(v40)--(v41)--(v42)--(v43)--(v33)--(v32)--(v22)--(v23)--(v13)--(v03)--(v93)--(v92)--(v91)--(v81)--(v82)--(v83)--(v73)--(v63)--(v53)--(v52)--(v51)--(v50)--(v60)--(v61)--(v62)--(v62)--(v72)--(v71)--(v70)--(v80)--(v90)--(v00);

\end{tikzpicture}
\hspace{0.5cm}
\begin{tikzpicture}[scale=0.8] 
\tikzstyle{vertex}=[circle, draw=black, thick, fill=black!15!white, inner sep = 0.6mm]
\foreach \x in {0,1,...,9} {
  \foreach \y in {0,1,2,3} {
    \node [vertex] (v\x\y) at (36*\x:1+0.6*\y) {};
  }
}
\draw [thick, dotted] (v00)--(v10)--(v20)--(v30)--(v40)--(v50)--(v60)--(v70)--(v80)--(v90)--(v00);
\draw [thick, dotted] (v03)--(v13)--(v23)--(v33)--(v43)--(v53)--(v63)--(v73)--(v83)--(v93)--(v03);

\foreach \x in {0,1,...,9} {
  \draw  [thick, dotted] (v\x0)--(v\x1)--(v\x2)--(v\x3);
}

\foreach \x/\y in {2/3, 4/5, 6/7, 8/9, 0/1} {
\draw [thick, dotted] (v\x1)--(v\y1);
\draw [thick, dotted] (v\x2)--(v\y2);
}

\foreach \x in {0,1,...,9} {
\draw [ultra thick] (v\x0)--(v\x1)--(v\x2)--(v\x3);
}

\foreach \x/\y in {0/1, 2/3, 4/5, 6/7, 8/9} {
\draw [ultra thick]  (v\x3)--(v\y3);
}

\foreach \x/\y in {1/2, 3/4, 5/6, 7/8, 9/0} {
\draw [ultra thick]  (v\x0)--(v\y0);
}

\end{tikzpicture}
\hspace{0.5cm}
\begin{tikzpicture}[scale=0.8] 
\tikzstyle{vertex}=[circle, draw=black, thick,  fill=black!15!white, inner sep = 0.6mm]
\foreach \x in {0,1,...,9} {
  \foreach \y in {0,1,2,3} {
    \node [vertex] (v\x\y) at (36*\x:1+0.6*\y) {};
  }
}
\draw [thick, dotted] (v00)--(v10)--(v20)--(v30)--(v40)--(v50)--(v60)--(v70)--(v80)--(v90)--(v00);
\draw [thick, dotted] (v03)--(v13)--(v23)--(v33)--(v43)--(v53)--(v63)--(v73)--(v83)--(v93)--(v03);

%\draw [ultra thick] (v00)--(v

\foreach \x in {0,1,...,9} {
  \draw [thick, dotted] (v\x0)--(v\x1)--(v\x2)--(v\x3);
}

\foreach \x/\y in {2/3, 4/5, 6/7, 8/9, 0/1} {
\draw [thick, dotted] (v\x1)--(v\y1);
\draw [thick, dotted] (v\x2)--(v\y2);
}

\draw [ultra thick] (v00)--(v10)--(v20)--(v30)--(v31)--(v21)--(v22)--(v32)--(v33)--(v23)--(v13)--(v12)--(v11)--(v01)--(v02)--(v03)--(v93)--(v92)--(v82)--(v83)--(v73)--(v63)--(v53)--(v43)--(v42)--(v52)--(v51)--(v41)--(v40)--(v50)--(v60)--(v61)--(v62)--(v72)--(v71)--(v70)--(v80)--(v81)--(v91)--(v90)--(v00);

\end{tikzpicture}

\end{center}
\caption{Types of Hamilton cycle in $\mathrm{RL}(5,4)$}
\label{fig:rl54hams}

\end{figure}
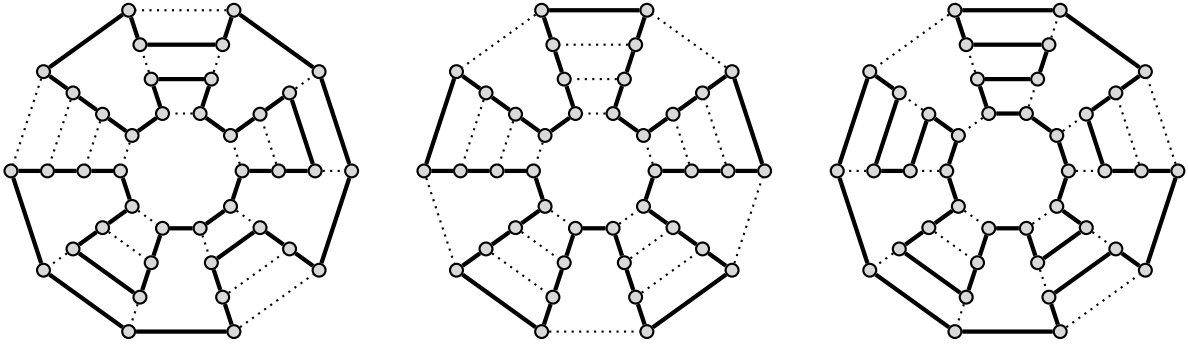

\begin{lemma}
The number of Hamilton cycles in an $m$-ring of $k$-ladders is:
\[
\mathrm{hc}(RL(m,k)) = 
\begin{cases}
2^m + m(k-1)^{m-1},& k \text{ odd}; \\
2 + m k (k-1)^{m-2}, & k \text{ even}.
\end{cases}
\]
\end{lemma}
\begin{proof}
Let $L_0$, $L_1$, $\ldots$, $L_{m-1}$ be the $k$-ladders in $RL(m,k)$ and distinguish the vertices of $L_i$ with a superscript $i$, so that $v_j^i$ is the copy of $v_j$ in $L_i$.  Let $E(L_i, L_{i+1})$ denote the pair of edges connecting $L_i$ to $L_{i+1}$, with all ladder indices taken modulo $m$. For any ladder index $i$, the four edges $E(L_{i-1},L_i) \cup E(L_i,L_{i+1})$ form an edge-cut of the graph and so any Hamilton cycle $C$ uses an even non-zero number of these edges, therefore either $2$ or $4$.
Parity considerations show that there are three distinct cases to consider, according to the number of edges that a Hamilton cycle $C$ uses from each pair $E(L_i, L_{i+1})$.  %Either $C$ uses exactly one edge from every pair $E(L_i, L_{i+1})$, or $C$ uses both edges from every pair $E(L_i, L_{i+1})$ or there is a unique index $j$ such that $C$ uses no edges from $E(L_j, L_{j+1})$ but uses both edges from every other pair $E(L_i, L_{i+1})$.

\smallskip
\noindent
{\bf Case 1:} $C$ uses both edges from each pair $E(L_i,L_{i+1})$ (first diagram in \cref{fig:rl54hams}).

For every $i$, the restriction of $C$ to $L_i$ either consists of two disjoint paths, one with endpoints $\{u_1^i, v_1^i\}$ and the second with endpoints $\{u_k^i, v_k^i\}$ or it consists of two paths with endpoints $\{u_1^i, u_k^i\}$ and $\{v_1^i, v_k^i\}$. In the first case, each of the two paths uses exactly one rung of $L_i$, and in order that these two paths use all the vertices of $L_i$, the two rungs are consecutive rungs in the ladder. In the second case, each path contains all the edges of one of the rails of the ladder and none of the rungs are used. These paths, along with the edges between consecutive ladders, form a Hamilton cycle if and only if there is exactly one index $j$ such that the restriction of $C$ to $L_j$ contains no rungs.  There are $m$ choices for this unique index $j$, but just one way to route the Hamilton cycle through $L_j$. For each of the $m-1$ other ladders, there are $k-1$ ways to select the two adjacent rungs to be used by $C$ and then the remaining edges are forced. This gives a total of $m(k-1)^{m-1}$ cycles of this type, independent of the parity of $k$.

\smallskip
\noindent
{\bf Case 2:} $C$ uses exactly one edge from each pair $E(L_i,L_{i+1})$ (second diagram in \cref{fig:rl54hams}).

If $C$ contains edges $(v_1^{i-1},u_1^i)$ and $(v_1^i,u_1^{i+1})$, then $C$ induces a $(u_1^i,v_1^i)$ Hamilton path in $L_i$. There is only one such Hamilton path of $L_i$, which uses all the rail edges and just the rung $(u_k^i, v_k^i)$. If $C$ contains edges $(v_1^{i-1},u_1^i)$ and $(v_k^i,u_k^{i+1})$, then $C$ induces a $(u_1^i, v_k^i)$ Hamilton path in $L_i$. If $k$ is even, then no such Hamilton path exists, while there is exactly one such Hamilton path if $k$ is odd.
Thus if $k$ is even then there are only two possibilities for $C$, one using every edge of the form $(v_1^{i-1},u_1^i)$ and $(v_1^i,u_1^{i+1})$ and the other using every  edge of the form $(v_k^{i-1},u_k^i)$ and $(v_k^i,u_k^{i+1})$. For odd $k$, there are $2^m$ possibilities for $C$, as an arbitrary choice of one of the two edges connecting each ladder to the next can be completed to a Hamilton cycle in a unique way. Thus there are either $2$ or $2^m$ Hamilton cycles of this type, depending on the parity of $k$.

\smallskip
\noindent
{\bf Case 3:} There is a unique index $j$ such that $C$ uses no edges from $E(L_j,L_{j+1})$, but both edges from every other pair $E(L_i,L_{i+1})$ with $i \ne j$ (third diagram in \cref{fig:rl54hams}).
%If $C$ is a Hamilton cycle, then it uses an even non-zero number of edges from every edge cut of $G$. For each $i$, let $\ell_i$ denote the number of edges of $C$ connecting $L_{i-1}$ and $L_i$, and let $r_i$ denote the number of edges of $C$ connecting $L_i$ to $L_{i+1}$  Then for each $i$, we have $(\ell_i, r_i) \in \{(1,1), (0,2), (2,0), (2,2)\}$. There are just three cases to consider --- either $(\ell_i, r_i) = (1,1)$ for all $i$, or $(\ell_i, r_i) = (2,2)$ for all $i$, or there is a distinguished index $j$ such that $(ell_j,r_j) = (2,0)$ and $(\ell_{j+1}, r_{j+1}) = (0,2)$ and all ot
%If any of the ladders have $(\ell_i, r_i) = (2,2)$, then there is exactly one index $j$ such that $C$ uses no edges between $L_j$ and $L_{j+1}$, but $C$ uses every other pair of edges between consecutive ladders.

In this case $C$ induces a $(v_1^j, v_k^j)$ Hamilton path in $L_j$ and a $(u_1^{j+1}, v_k^{j+1})$ Hamilton path in $L_{j+1}$. If $k$ is odd, then there is no such Hamilton path, but if $k$ is even, there is a unique suitable Hamilton path obtained by alternating rung-edge and rail-edge from the endpoints. For every ladder other than $L_j$ and $L_{j+1}$, the cycle $C$ induces a $(u_1^i, v_1^i)$ path and a disjoint $(u_k^i, v_k^i)$ path, and by the argument from Case 1, there are $k-1$ distinct choices for this pair of paths. Therefore the total number of choices for Hamilton cycles of this type is $m(k-1)^{m-2}$ if $k$ is even, and $0$ if $k$ is odd.

\smallskip
\noindent
Adding the numbers of Hamilton cycles of each type gives the stated result. 
\end{proof}

If the number of vertices is a multiple of $8$, say $8s$, then an $s$-ring of $4$-ladders has $2 + (4 s) 3^{s-2}$ Hamilton cycles, which we conjecture to be the maximum for sufficiently large $v$. In fact at $v=40$ the graph $RL(5,4)$ with $542$ Hamilton cycles is the unique extremal graph, and we expect this to continue for $v \geqslant 48$. However for $16$, $24$ and $32$ vertices there are seemingly sporadic graphs with more Hamilton cycles than the ring of $4$-ladders.

\section{Planar cubic cyclically \texorpdfstring{$5$}{5}-edge-connected graphs with few Hamilton Cycles}
\label{cub5confew}

After the expected initial variability, the data from \cref{tab:plancubcyc5con} shows that the minimum number of Hamilton cycles appears to settle down as the number of vertices increases. It is therefore tempting and plausible to conjecture that this is the minimum and that every Hamiltonian graph in this class has at least $16$ cycles. 

In total, there are $18$ planar cubic cyclically $5$-edge-connected graphs on $52$--$64$ vertices with exactly $16$ Hamilton cycles. (There are $3$, $2$, $3$, $1$, $1$, $2$ and $6$ graphs on $52$, $54$, $56$, $58$, $60$, $62$ and $64$ vertices respectively). However while they all look qualitatively ``similar'' to each other, it is difficult to discern any structural patterns with sufficient precision to be generalized.

Having nothing more to add, we therefore conclude this section by presenting two graphs from this collection in \cref{fig:plan5confew} in the hope that some reader of this paper can detect a pattern that has eluded us.

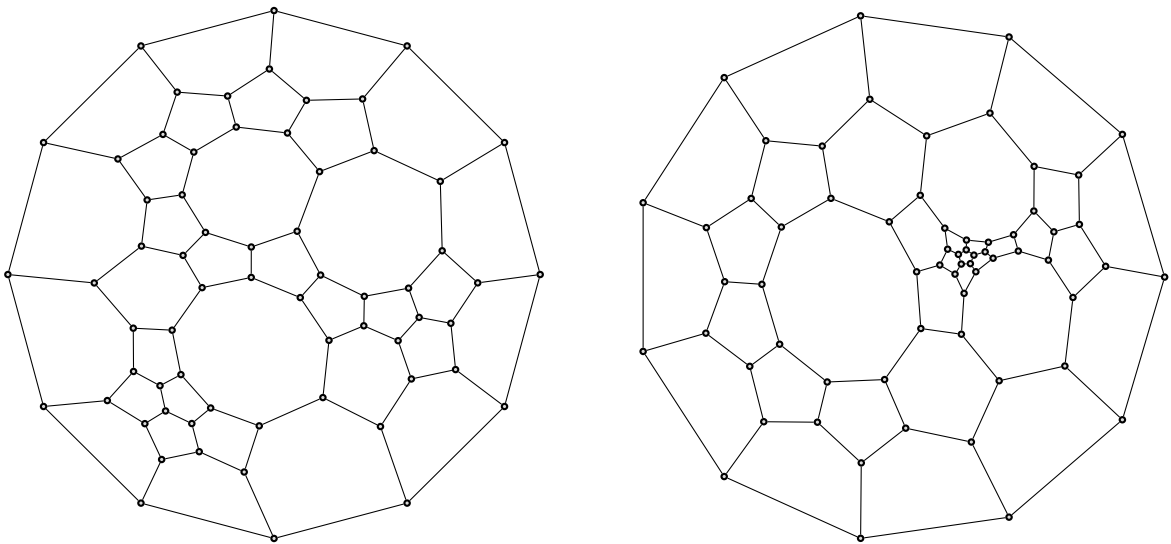
\begin{figure}
\begin{center}
\begin{tikzpicture}[scale=0.35]
\tikzstyle{vertex}=[circle, draw=black, thick, fill=black!15!white, inner sep = 0.25mm]
\node [vertex] (v0) at (-5.867,4.378) {};
\node [vertex] (v1) at (-8.660,5.000) {};
\node [vertex] (v2) at (-4.767,2.825) {};
\node [vertex] (v3) at (-4.174,5.309) {};
\node [vertex] (v4) at (-5.000,8.660) {};
\node [vertex] (v5) at (-10.000,0.000) {};
\node [vertex] (v6) at (-4.981,1.077) {};
\node [vertex] (v7) at (-3.452,3.020) {};
\node [vertex] (v8) at (-3.015,4.639) {};
\node [vertex] (v9) at (-3.639,6.910) {};
\node [vertex] (v10) at (0.000,10.000) {};
\node [vertex] (v11) at (-8.660,-5.000) {};
\node [vertex] (v12) at (-6.756,-0.319) {};
\node [vertex] (v13) at (-3.420,0.725) {};
\node [vertex] (v14) at (-2.576,1.596) {};
\node [vertex] (v15) at (-1.419,5.588) {};
\node [vertex] (v16) at (-1.744,6.762) {};
\node [vertex] (v17) at (-0.175,7.787) {};
\node [vertex] (v18) at (5.000,8.660) {};
\node [vertex] (v19) at (-5.000,-8.660) {};
\node [vertex] (v20) at (-6.265,-4.775) {};
\node [vertex] (v21) at (-5.289,-2.033) {};
\node [vertex] (v22) at (-2.702,-0.497) {};
\node [vertex] (v23) at (-0.855,1.042) {};
\node [vertex] (v24) at (0.504,5.362) {};
\node [vertex] (v25) at (1.218,6.600) {};
\node [vertex] (v26) at (3.327,6.651) {};
\node [vertex] (v27) at (8.660,5.000) {};
\node [vertex] (v28) at (0.000,-10.000) {};
\node [vertex] (v29) at (-4.222,-7.008) {};
\node [vertex] (v30) at (-4.854,-5.650) {};
\node [vertex] (v31) at (-5.279,-3.673) {};
\node [vertex] (v32) at (-3.830,-2.108) {};
\node [vertex] (v33) at (-0.857,-0.110) {};
\node [vertex] (v34) at (0.868,1.640) {};
\node [vertex] (v35) at (1.711,3.899) {};
\node [vertex] (v36) at (3.761,4.694) {};
\node [vertex] (v37) at (6.247,3.531) {};
\node [vertex] (v38) at (10.000,0.000) {};
\node [vertex] (v39) at (5.000,-8.660) {};
\node [vertex] (v40) at (-1.122,-7.482) {};
\node [vertex] (v41) at (-2.811,-6.712) {};
\node [vertex] (v42) at (-4.076,-5.169) {};
\node [vertex] (v43) at (-4.285,-4.212) {};
\node [vertex] (v44) at (-3.499,-3.792) {};
\node [vertex] (v45) at (0.985,-0.874) {};
\node [vertex] (v46) at (1.749,-0.020) {};
\node [vertex] (v47) at (6.318,0.900) {};
\node [vertex] (v48) at (7.654,-0.316) {};
\node [vertex] (v49) at (8.660,-5.000) {};
\node [vertex] (v50) at (3.999,-5.759) {};
\node [vertex] (v51) at (-0.556,-5.734) {};
\node [vertex] (v52) at (-3.090,-5.646) {};
\node [vertex] (v53) at (-2.381,-5.057) {};
\node [vertex] (v54) at (2.065,-2.493) {};
\node [vertex] (v55) at (3.392,-0.826) {};
\node [vertex] (v56) at (5.054,-0.517) {};
\node [vertex] (v57) at (6.643,-1.847) {};
\node [vertex] (v58) at (6.821,-3.601) {};
\node [vertex] (v59) at (5.161,-3.956) {};
\node [vertex] (v60) at (1.836,-4.662) {};
\node [vertex] (v61) at (3.373,-1.942) {};
\node [vertex] (v62) at (5.453,-1.624) {};
\node [vertex] (v63) at (4.662,-2.507) {};
\draw (v0)--(v1);
\draw (v0)--(v2);
\draw (v0)--(v3);
\draw (v1)--(v4);
\draw (v1)--(v5);
\draw (v2)--(v6);
\draw (v2)--(v7);
\draw (v3)--(v8);
\draw (v3)--(v9);
\draw (v4)--(v9);
\draw (v4)--(v10);
\draw (v5)--(v11);
\draw (v5)--(v12);
\draw (v6)--(v12);
\draw (v6)--(v13);
\draw (v7)--(v8);
\draw (v7)--(v14);
\draw (v8)--(v15);
\draw (v9)--(v16);
\draw (v10)--(v17);
\draw (v10)--(v18);
\draw (v11)--(v19);
\draw (v11)--(v20);
\draw (v12)--(v21);
\draw (v13)--(v14);
\draw (v13)--(v22);
\draw (v14)--(v23);
\draw (v15)--(v16);
\draw (v15)--(v24);
\draw (v16)--(v17);
\draw (v17)--(v25);
\draw (v18)--(v26);
\draw (v18)--(v27);
\draw (v19)--(v28);
\draw (v19)--(v29);
\draw (v20)--(v30);
\draw (v20)--(v31);
\draw (v21)--(v31);
\draw (v21)--(v32);
\draw (v22)--(v32);
\draw (v22)--(v33);
\draw (v23)--(v33);
\draw (v23)--(v34);
\draw (v24)--(v25);
\draw (v24)--(v35);
\draw (v25)--(v26);
\draw (v26)--(v36);
\draw (v27)--(v37);
\draw (v27)--(v38);
\draw (v28)--(v39);
\draw (v28)--(v40);
\draw (v29)--(v30);
\draw (v29)--(v41);
\draw (v30)--(v42);
\draw (v31)--(v43);
\draw (v32)--(v44);
\draw (v33)--(v45);
\draw (v34)--(v35);
\draw (v34)--(v46);
\draw (v35)--(v36);
\draw (v36)--(v37);
\draw (v37)--(v47);
\draw (v38)--(v48);
\draw (v38)--(v49);
\draw (v39)--(v49);
\draw (v39)--(v50);
\draw (v40)--(v41);
\draw (v40)--(v51);
\draw (v41)--(v52);
\draw (v42)--(v43);
\draw (v42)--(v52);
\draw (v43)--(v44);
\draw (v44)--(v53);
\draw (v45)--(v46);
\draw (v45)--(v54);
\draw (v46)--(v55);
\draw (v47)--(v48);
\draw (v47)--(v56);
\draw (v48)--(v57);
\draw (v49)--(v58);
\draw (v50)--(v59);
\draw (v50)--(v60);
\draw (v51)--(v53);
\draw (v51)--(v60);
\draw (v52)--(v53);
\draw (v54)--(v60);
\draw (v54)--(v61);
\draw (v55)--(v56);
\draw (v55)--(v61);
\draw (v56)--(v62);
\draw (v57)--(v58);
\draw (v57)--(v62);
\draw (v58)--(v59);
\draw (v59)--(v63);
\draw (v61)--(v63);
\draw (v62)--(v63);
\end{tikzpicture}
\hspace{1cm}
\begin{tikzpicture}[scale=0.35]
\tikzstyle{vertex}=[circle, draw=black, thick, fill=black!15!white, inner sep = 0.25mm]
\node [vertex] (v0) at (2.125,0.111) {};
\node [vertex] (v1) at (2.359,0.493) {};
\node [vertex] (v2) at (1.552,0.456) {};
\node [vertex] (v3) at (2.465,-0.617) {};
\node [vertex] (v4) at (2.700,0.509) {};
\node [vertex] (v5) at (2.251,0.860) {};
\node [vertex] (v6) at (1.847,1.056) {};
\node [vertex] (v7) at (0.682,0.200) {};
\node [vertex] (v8) at (2.363,-2.165) {};
\node [vertex] (v9) at (2.907,0.202) {};
\node [vertex] (v10) at (2.834,0.832) {};
\node [vertex] (v11) at (2.546,1.031) {};
\node [vertex] (v12) at (1.740,1.852) {};
\node [vertex] (v13) at (-0.346,2.094) {};
\node [vertex] (v14) at (0.841,-1.949) {};
\node [vertex] (v15) at (3.783,-3.928) {};
\node [vertex] (v16) at (3.555,0.715) {};
\node [vertex] (v17) at (3.256,0.956) {};
\node [vertex] (v18) at (2.555,1.401) {};
\node [vertex] (v19) at (0.818,3.099) {};
\node [vertex] (v20) at (-2.538,2.984) {};
\node [vertex] (v21) at (-0.522,-3.883) {};
\node [vertex] (v22) at (2.737,-6.249) {};
\node [vertex] (v23) at (6.251,-3.369) {};
\node [vertex] (v24) at (4.503,0.987) {};
\node [vertex] (v25) at (3.378,1.320) {};
\node [vertex] (v26) at (1.059,5.350) {};
\node [vertex] (v27) at (-2.866,4.962) {};
\node [vertex] (v28) at (-4.402,1.896) {};
\node [vertex] (v29) at (-2.678,-3.975) {};
\node [vertex] (v30) at (0.272,-5.724) {};
\node [vertex] (v31) at (4.154,-9.096) {};
\node [vertex] (v32) at (8.413,-5.406) {};
\node [vertex] (v33) at (6.556,-0.773) {};
\node [vertex] (v34) at (5.633,0.643) {};
\node [vertex] (v35) at (4.322,1.604) {};
\node [vertex] (v36) at (3.437,6.213) {};
\node [vertex] (v37) at (-1.077,6.737) {};
\node [vertex] (v38) at (-4.983,5.166) {};
\node [vertex] (v39) at (-5.535,2.979) {};
\node [vertex] (v40) at (-5.133,-0.274) {};
\node [vertex] (v41) at (-4.466,-2.544) {};
\node [vertex] (v42) at (-3.046,-5.498) {};
\node [vertex] (v43) at (-1.399,-7.040) {};
\node [vertex] (v44) at (-1.423,-9.898) {};
\node [vertex] (v45) at (10.000,0.000) {};
\node [vertex] (v46) at (7.784,0.408) {};
\node [vertex] (v47) at (5.839,1.714) {};
\node [vertex] (v48) at (5.086,2.504) {};
\node [vertex] (v49) at (5.097,4.194) {};
\node [vertex] (v50) at (4.154,9.096) {};
\node [vertex] (v51) at (-1.423,9.898) {};
\node [vertex] (v52) at (-6.549,7.557) {};
\node [vertex] (v53) at (-7.220,1.874) {};
\node [vertex] (v54) at (-6.530,-0.175) {};
\node [vertex] (v55) at (-5.588,-3.383) {};
\node [vertex] (v56) at (-5.061,-5.480) {};
\node [vertex] (v57) at (-6.549,-7.557) {};
\node [vertex] (v58) at (8.413,5.406) {};
\node [vertex] (v59) at (6.797,1.996) {};
\node [vertex] (v60) at (6.769,3.866) {};
\node [vertex] (v61) at (-9.595,2.817) {};
\node [vertex] (v62) at (-7.238,-2.125) {};
\node [vertex] (v63) at (-9.595,-2.817) {};
\draw (v0)--(v1);
\draw (v0)--(v2);
\draw (v0)--(v3);
\draw (v1)--(v4);
\draw (v1)--(v5);
\draw (v2)--(v6);
\draw (v2)--(v7);
\draw (v3)--(v8);
\draw (v3)--(v9);
\draw (v4)--(v9);
\draw (v4)--(v10);
\draw (v5)--(v6);
\draw (v5)--(v11);
\draw (v6)--(v12);
\draw (v7)--(v13);
\draw (v7)--(v14);
\draw (v8)--(v14);
\draw (v8)--(v15);
\draw (v9)--(v16);
\draw (v10)--(v11);
\draw (v10)--(v17);
\draw (v11)--(v18);
\draw (v12)--(v18);
\draw (v12)--(v19);
\draw (v13)--(v19);
\draw (v13)--(v20);
\draw (v14)--(v21);
\draw (v15)--(v22);
\draw (v15)--(v23);
\draw (v16)--(v17);
\draw (v16)--(v24);
\draw (v17)--(v25);
\draw (v18)--(v25);
\draw (v19)--(v26);
\draw (v20)--(v27);
\draw (v20)--(v28);
\draw (v21)--(v29);
\draw (v21)--(v30);
\draw (v22)--(v30);
\draw (v22)--(v31);
\draw (v23)--(v32);
\draw (v23)--(v33);
\draw (v24)--(v34);
\draw (v24)--(v35);
\draw (v25)--(v35);
\draw (v26)--(v36);
\draw (v26)--(v37);
\draw (v27)--(v37);
\draw (v27)--(v38);
\draw (v28)--(v39);
\draw (v28)--(v40);
\draw (v29)--(v41);
\draw (v29)--(v42);
\draw (v30)--(v43);
\draw (v31)--(v32);
\draw (v31)--(v44);
\draw (v32)--(v45);
\draw (v33)--(v34);
\draw (v33)--(v46);
\draw (v34)--(v47);
\draw (v35)--(v48);
\draw (v36)--(v49);
\draw (v36)--(v50);
\draw (v37)--(v51);
\draw (v38)--(v39);
\draw (v38)--(v52);
\draw (v39)--(v53);
\draw (v40)--(v41);
\draw (v40)--(v54);
\draw (v41)--(v55);
\draw (v42)--(v43);
\draw (v42)--(v56);
\draw (v43)--(v44);
\draw (v44)--(v57);
\draw (v45)--(v46);
\draw (v45)--(v58);
\draw (v46)--(v59);
\draw (v47)--(v48);
\draw (v47)--(v59);
\draw (v48)--(v49);
\draw (v49)--(v60);
\draw (v50)--(v51);
\draw (v50)--(v58);
\draw (v51)--(v52);
\draw (v52)--(v61);
\draw (v53)--(v54);
\draw (v53)--(v61);
\draw (v54)--(v62);
\draw (v55)--(v56);
\draw (v55)--(v62);
\draw (v56)--(v57);
\draw (v57)--(v63);
\draw (v58)--(v60);
\draw (v59)--(v60);
\draw (v61)--(v63);
\draw (v62)--(v63);
\end{tikzpicture}
\end{center}
\caption{Cyclically $5$-edge-connected graphs on $64$ vertices with  $16$ Hamilton cycles}
\label{fig:plan5confew}
\end{figure}

\section{Planar cubic cyclically \texorpdfstring{$5$}{5}-edge-connected graphs with many Hamilton Cycles}
\label{cub5conmany}

Problem 5 from Chia and Thomassen \cite{MR2951794} asks if the generalized Petersen graphs $P(2n,2)$ on $4n$ vertices
have the greatest number of Hamilton cycles amongst all planar cubic cyclically $5$-edge-connected graphs on the same
number of vertices. In fact, in the text of their paper they speculate on something rather stronger, namely that \emph{``For cubic cyclically $5$-edge-connected graphs, the generalized Petersen graphs are possible candidates for those that have the maximum number of longest cycles''}, whereas in Problem 5 the same question is raised only with respect to \emph{planar} graphs.  

For $m \geqslant 9$, the generalized Petersen graph $P(m,2)$ on $2m$ vertices is a cyclically $5$-edge-connected
cubic graph. It is planar if and only if $m$ is even, and in this case the number of Hamilton cycles of $P(m,2)$ is given by
\[
\text{hc}(P(m,2)) = 
\begin{cases}
2\left( F_{\frac{m}{2}+1} + F_{\frac{m}{2}-1} - 1\right),  & m \equiv 0, 2 \pmod{6}; \\
2\left( F_{\frac{m}{2}+1} + F_{\frac{m}{2}-1} - 1\right) + m , & m \equiv 4 \pmod{6},
\end{cases}
\]
where $F_m$ is the $m$-th Fibonacci number (Schwenk \cite{MR1007713}) For example, the dodecahedron is isomorphic to $P(10,2)$, which has $10+2(F_6+F_4 -1) = 30$ Hamilton cycles. 

The computational results for $v \leqslant 64$ summarised in Table~\ref{tab:plancubcyc5con} 
reveal that $P(2n,2)$ is indeed frequently the graph with the most Hamilton cycles 
when the number of vertices is a multiple of $4$, but with a number of curious exceptions on $v=40$, $v=56$ and $v=60$ vertices.  In the remainder of this section we describe an infinite family of graphs, one for each multiple of $20$ vertices, that includes the exceptional graphs on $40$ and $60$ vertices.  On $20$ vertices, this graph coincides with $P(10,2)$, but for all larger multiples of $20$, this graph has strictly more Hamilton cycles than the generalized Petersen graph $P(2n,2)$ on the same number of vertices.

We define a graph, which we denote $N(5,k)$ (for reasons to be explained below) in the following fashion: arrange in a concentric manner $k+2$ disjoint cycles, where the innermost and outermost cycles are copies of $C_5$ and the remaining $k$ cycles are copies of $C_{10}$, as illustrated in the left-hand graph of Figure~\ref{fig:x3circular}. Then connect each consecutive pair of cycles with a $5$-vertex matching in such a way that the vertices on each of the inner $10$-cycles are connected alternately to the neighbouring cycle on the inside, and the neighbouring cycle on the outside, while respecting the cyclic order of the vertices on all the cycles. (A glance at the second illustration of Figure~\ref{fig:x3circular} will clarify this rather cumbersome definition!) 
The graph $N(5,1)$ has two $5$-cycles separated by a single $10$-cycle and is easily recognised as the dodecahedron, while $N(5,3)$ and $N(5,5)$ are the graphs with the greatest number of Hamilton cycles on $40$ and $60$ vertices respectively. 

Statistical physicists and others interested in the chromatic roots of graphs may recognise this graph as the \emph{planar dual} of an $n_F \times 5_P$ strip of the triangular lattice with a vertex of degree $5$ as the end-graph at each end.  In fact, we will count the Hamilton cycles of $N(5,k)$ in a manner highly reminiscent of the ``transfer matrix'' techniques used in those fields.
On the other hand, mathematical chemists may recognise this graph as a \emph{fullerene graph}, which is a cubic planar graph all of whose faces have size $5$ or $6$. This name arose because certain carbon allotropes named \emph{fullerenes} have a carbon skeleton (i.e., the configuration of carbon atoms joined by bonds) that is a fullerene graph. The family $\mathcal{N} = \{N(5,k) \mid k \geqslant 1\}$ corresponds to a family of molecules called \emph{nanotubes}, because their physical realization is a long thin cylindrical structure. So the notation $N(5,k)$ is a mnemonic for ``nanotube of width $5$ and length $k$''. (There is more precise chemical nomenclature as there are a number of different types of nanotube, but we do not need this for our current purpose.)

Although these circular pictures highlight the symmetry and planarity of the graphs $N(5,k)$, in order to analyse their Hamilton cycles we will draw them in a way that exhibits the $5$-edge cuts separating consecutive layers. Figure~\ref{fig:x3linear} shows the graph $N(5,3)$ with the two $5$-cycles at the left- and right-hand end of the picture, and with three ``$10$-vertex layers'' between them. Our counting argument is based on examining which configurations of edges inside a layer can be part of a Hamilton cycle, and in how many ways these configurations can be combined across layers. Figure~\ref{fig:closeup} shows a representation of an arbitrary layer, consisting of five edges $\{e_0, e_1, \ldots, e_4\}$ on the left-hand side of the layer, five edges $\{f_0, f_1, \ldots, f_4\}$ on the right-hand side and the ten vertices $\{v_0, v_1, \ldots, v_4, w_0, w_1, \ldots, w_4\}$ in the layer itself. This configuration exhibits an obvious circular symmetry of order $5$, generated by mapping  all subscripts $i \to i + 1 \pmod{5}$ simultaneously. A Hamilton cycle uses an even number of edges from any edge-cut, and as every edge within a layer connects a vertex in  
$\{v_0,v_1,\ldots,v_4\}$ to a vertex in $\{w_0,w_1,\ldots,w_4\}$, the number of edges used by a Hamilton cycle is the same on each side of the layer. So there are two types of Hamilton cycle --- those using $2$ edges between each consecutive pair of layers, and those using $4$ edges between each consecutive pair of layers. We call these Type 2 and Type 4 respectively.

It turns out that $N(5,k)$ always has Type 2 Hamilton cycles, but it only has Type 4 Hamilton cycles when $k$ is odd, and so it is only when the number of vertices is a multiple of $20$ that $N(5,k)$ has many Hamilton cycles. 

\begin{theorem}
For odd $k$, the graph $N(5,k)$ on $n=10+10k$ vertices has 
\[5 \cdot 2^{k} + 20 \cdot 12^{(k-1)/2}\] 
Hamilton cycles, while for even $k$ it has $5 \cdot 2^k$ Hamilton cycles.
\end{theorem}

\begin{proof}
For brevity, let $N$ denote the graph $N(5,k)$ and assume that it is laid out as in Figure~\ref{fig:x3linear} with $L_0$ denoting the left-hand $5$-cycle (here we are using $L$ for ``layer''), $L_1$, $L_2$, $\ldots$, $L_k$ denoting the $10$-cycles in order from left-to-right, and $L_{k+1}$ denoting the right-hand $5$-cycle.  We call $L_0$ the \emph{starting layer}, $L_i$ for $i \leqslant i \leqslant k$ the \emph{internal layers} and $L_{k+1}$ the \emph{finishing layer}. Given a Hamilton cycle $H$, let $H(i)$ denote the restriction of $H$ to $L_i$, and let $H(0,i)$ denote the restriction of $H$ to the vertices $L_0 \cup L_1 \cup \cdots \cup L_i$. We say that $H(0,i)$ is a \emph{partial Hamilton cycle with $i+1$ layers}. Our overall strategy will be to count  how many ways there are to extend a partial Hamilton cycle with $i+1$ layers to a partial Hamilton cycle with $i+2$ layers.
As $H$ is a Hamilton cycle, $H(0,i)$ is a disjoint union of paths that collectively touch every vertex in $L_0 \cup L_1 \cup \ldots \cup L_i$, and whose vertices of degree $1$, which we call \emph{terminals}, are all on the right-hand side of $L_i$. Illustrating this, 
\cref{fig:partialhamcyc} shows a partial Hamilton cycle $H(0,2)$ in $N(5,3)$ consisting of two paths, and which has terminals $\{w_0, w_1, w_2, w_3\}$. The fifth vertex $w_4$ is not a terminal because it already has degree $2$ in the partial Hamilton cycle. Each path in the partial Hamilton cycle has two endvertices, and so the terminals of $H(0,i)$ are grouped into pairs. Therefore the partial Hamilton cycle determines a \emph{partition} of  $\{w_0, w_1, \ldots, w_4\}$ into pairs, and singletons (the vertices of degree $2$), which we will call the \emph{terminal partition} of $H(0,i)$. 
As an example, the  partial Hamilton cycle of \cref{fig:partialhamcyc} has terminal partition $\{w_0, w_3 \mid w_1, w_2 \mid w_4\}$ which for brevity we will denote $\{03|12|4\}$.

The key observation behind our counting technique is that it is \emph{the terminal partition alone} that determines in how many ways a partial Hamilton cycle can be extended by another layer and, in addition, the terminal partitions of the resulting extended partial Hamilton cycles. Thus our strategy will be to work layer-by-layer, keeping track at each stage not only of the total number of partial Hamilton cycles, but also the auxiliary information of how many partial Hamilton cycles have each type of terminal partition. By the observation above, this additional information is sufficient to calculate both the number of partial Hamilton cycles \emph{and} the additional information for the partial Hamilton cycles with an additional layer.    %(For those familiar with the relevant literature, this is entirely analogous to the technique of computing chromatic polynomials of lattice graphs by counting the number of connected components induced by every possible subset of the edges, grouping them into classes according to the partition 
%at each stage counting in how many ways a partial Hamilton cycle with a given terminal partition can be extended by one layer to a larger partial Hamilton cycle, and then finally completed to a Hamilton cycle. The edges between $L_i$ and $L_{i+1}$ in any extension of $H(0,i)$ are already determined --- if a vertex $v$ is a terminal of $H(0,i)$, then the unique edge connecting $v$ to $L_{i+1}$ must be in $H(0,i+1)$. Therefore it suffices to consider all possible configurations of edges within $L_{i+1}$,  decide how many of them lead to partial Hamilton cycles, and keep track of the terminal partitions of the newly extended partial Hamilton cycles.

The restriction $H(i)$ of a Hamilton cycle $H$ to any individual internal layer $L_i$ has the following structure: it is a spanning subgraph of $L_i$ where every vertex has degree $1$ or $2$, and which has an even number of terminals in $\{v_0, v_1, \ldots, v_4\}$ and (as a consequence) the same even number of terminals in  $\{w_0, w_1, \ldots, w_4\}$.  Similarly, $H(0)$ and $H(k+1)$ are spanning subgraphs of $L_0$ and $L_{k+1}$ respectively,  with every vertex having degree $1$ or $2$, and again having an even number of vertices of degree $1$ that we call terminals. 
It is clear that the sequence $(H(0), H(1), H(2), \ldots, H(k+1))$ determines the entire Hamilton cycle because every terminal vertex is incident with a unique edge connecting consecutive layers. 

We will call a spanning subgraph of $C_{10}$ (labelled as in \cref{fig:closeup}) that meets these conditions an \emph{internal tile} and a spanning subgraph of $C_5$ meeting these conditions an \emph{end tile}. So each Hamilton cycle can be identified with a sequence of tiles, starting and ending with an end-tile, and otherwise consisting of internal tiles. In order that a sequence of tiles correspond to a Hamilton cycle it is necessary that the terminals on one tile match up to the terminals on adjacent tiles, and also that the resulting $2$-regular graph has just one component. We say that a tile $T$ is \emph{consistent} with a partial Hamilton cycle $H(0,i)$ provided that adding $T$ to the sequence of tiles yields a valid partial Hamilton cycle.

Now we are ready to count the Type 2 Hamilton cycles. A routine case analysis shows that every end tile is equivalent to the one shown in Figure~\ref{fig:type2tiles} and every internal tile is equivalent to one of the two shown in Figure~\ref{fig:type2tiles}. Thus up to symmetry we may assume that $H(0)$ is  the tile pictured in Figure~\ref{fig:type2tiles}. The only valid choices for the next tile are those that have $\{v_2, v_3\}$ as left-hand terminals, so there are two choices for $H(1)$ (also shown in Figure~\ref{fig:type2tiles}).  Now although the exact choices for $H(2)$ depend on which of the two choices for $H(1)$ was used, the \emph{number of possibilities} does not. All of the tiles have two consecutive (in the cyclic order) vertices for left-hand terminals and two consecutive vertices for their right-hand terminals, and so there are always exactly two tiles consistent with any partial Hamilton cycle.  So in total, there are $5$ choices for the $5$-tile $H(0)$, then $2$ choices for each of $H(1)$, $H(2)$, $\ldots$, $H(k)$ and  finally there is a unique choice of final tile to complete the Hamilton cycle. Therefore for all $k$ there are exactly $5 \times 2^k$ Hamilton cycles of Type 2.

For the Hamilton cycles of Type 2, there was only one requirement when extending the tuple of tiles, namely that the right- and left-hand terminals of consecutive tiles must line up. Therefore the choices for each tile depended purely on the (right-hand) terminals of the previous one. However the analysis for Type 4 is slightly more subtle, because whether a tile is consistent or otherwise depends not just on the set of all terminals, but also on the \emph{terminal partition} of $H(0,i)$. If the new edges create a path in $H(0,i+1)$ joining two terminals of $H(0,i)$ in the same cell, then  $H(0,i+1)$ will contain a short cycle, so never lead to a Hamilton cycle. If the new tile is consistent with $H(0,i)$, then by adding it we obtain a partial Hamilton cycle $H(0,i+1)$ and this newly extended partial Hamilton cycle will have its own terminal partition. For example, \cref{fig:extending} shows a possible partial Hamilton cycle $H(0,1)$ with terminal partition $\{04|12|3\}$ and a possible extension $H(0,2)$ which has terminal partition $\{03|12|4\}$ (the terminal partition is illustrated by double-ended arrows showing the ends of each $2$-element cell). In this situation we say that the tile $H(2)$ \emph{transfers} the terminal partition $\{04|12|3\}$ to the terminal partition $\{03|12|4\}$.

As $N(5,k)$ is planar, the only terminal partitions that can arise are \emph{non-crossing} partitions, and there are ten non-crossing partitions with one singleton and two pairs. These ten partitions fall into two orbits under the cyclic automorphism of order $5$, namely $\{01 | 2 | 34\}$ and its five rotations, and $\{04 | 13 | 2\}$ and its five rotations, which we will denote as $\mathcal{O}_1$ and $\mathcal{O}_2$ respectively. 
Simple case analysis shows that up to rotation, there is a unique possible starting tile (i.e. choice for $H(0)$)  as illustrated in the first subfigure of \cref{fig:compat0}. The corresponding partial Hamilton cycle has terminal partition $\{04|13|2\}$ which lies in $\mathcal{O}_2$.
There are exactly four tiles compatible with this terminal partition, as illustrated in the second subfigure of \cref{fig:compat0} and the three subfigures of \cref{fig:compat1}.  These four tiles transfer the terminal partition $\{04|13|2\}$ to $\{01|2|34\}$, $\{01|2|34\}$, $\{04|1|23\}$ and $\{01|23|4\}$ respectively, all of which lie in the orbit $\mathcal{O}_1$.  Similarly, given a terminal partition from  $\mathcal{O}_1$, say $\{01|2|34\}$, there are three tiles compatible with this, as illustrated in \cref{fig:compat2}. These three tiles transfer the terminal partition $\{01|2|34\}$ to the terminal partitions
$\{04 | 13 | 2\}$, $\{0|14|23\}$ and $\{01|24|3\}$ respectively, which all belong to $\mathcal{O}_2$.

Now we can count the number of partial Hamilton cycles $H(0,i)$ for any $i$. There are exactly $5$ choices for $H(0)$, each of which have a terminal partition from $\mathcal{O}_2$. Each of these can be extended in $4$ ways, giving $5 \times 4 = 20$ choices for $H(0,1)$, each of which has a terminal partition from $\mathcal{O}_1$. Each of these can be extended in $3$ ways, given $5 \times 4 \times 3 = 60$ choices for $H(0,2)$, each of which has a terminal partition from $\mathcal{O}_2$.   Therefore there are $4 \times 3 = 12$ ways to extend a given partial Hamilton cycle $H(0,i)$ by two additional layers to a partial Hamilton cycle $H(0,i+2)$.

Finally we must consider how a partial Hamilton cycle $H(0,k)$ can be completed to a Hamilton cycle, by adding the final tile $H(k+1)$. Again, a simple case analysis on the five-cycle shows that this is only possible if the terminal partition of $H(0,k)$ lies in $\mathcal{O}_1$, and in this case there is a unique way to complete the partial Hamilton cycle.   In particular if $k$ is even, then the terminal partition of $H(0,k)$ lies in $\mathcal{O}_2$ and so it cannot be completed, meaning that these graphs have no Type 4 Hamilton cycles.  If $k$ is odd however, then there are $5$ choices for $H(0)$, $4$ choices for $H(1)$, and then $12$ choices for the remaining $(k-1)/2$ pairs of layers, giving a total of $20 \cdot 12^{(k-1)/2}$ Hamilton cycles of Type 4.

Adding the contributions from each type of Hamilton cycle yields the stated result.  \end{proof}

The nanotubes $N(5,k)$ have many Hamilton cycles only when $k$ is odd, so the number of vertices is a multiple of $20$, say $v = 20t$. Then we are comparing the generalized Petersen graph $P(10 t, 2)$ and the nanotube $N(5, 2 t -1)$. Evaluating the two expressions for the numbers of Hamilton cycles in each graph and extracting the leading term gives the following values:
\[
\text{hc}(P(10t,2)) \approx 2 (\varphi^5)^t \quad \text{ and } \quad \text{hc}(N(5,2t-1)) \approx \left(\frac{5}{3}\right) 12^t
\]
where $\varphi = (1+\sqrt{5})/2$ is the Golden Ratio.

As $\varphi^5 \approx 11.0902 < 12$ it follows that asymptotically the nanotubes have more Hamilton cycles than the equal-sized generalized Petersen graph, thereby answering Chia and Thomassen's question in the affirmative. Checking the exact values for small $t$ confirms that when $t=2$, the nanotube on $40$ vertices already has strictly more Hamilton cycles than $P(20,2)$ and this holds for all larger values of $t$.

%So if $v = 20 t$ then we are comparing the number of Hamilton cycles in $P(10 t, 2)$ and $N(5,2 t-1)$.

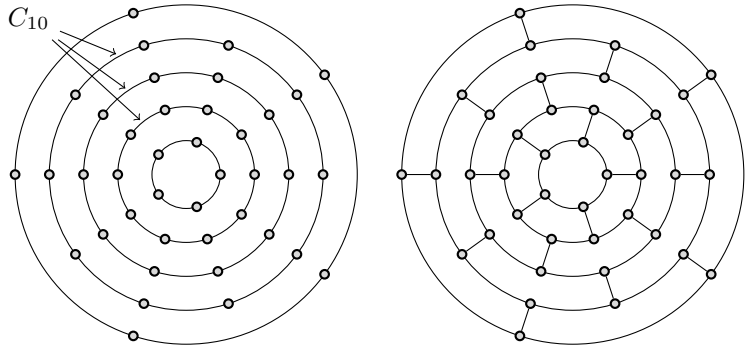
\begin{figure}
\begin{center}
\begin{tikzpicture}[scale=0.9]
\tikzstyle{vertex}=[circle, draw=black, thick, fill=black!15!white, inner sep = 0.4mm]
\draw (0,0) circle (0.5cm);
\draw (0,0) circle (1cm);
\draw (0,0) circle (1.5cm);
\draw (0,0) circle (2cm);
\draw (0,0) circle (2.5cm);
\foreach \x in {0,1,...,4} {
\node [vertex] (v\x) at (\x*360/5:0.5cm) {};
\node [vertex] (w\x) at (36+\x*360/5:2.5cm) {};
 }
\foreach \x in {0,1,...,9} {
\node[vertex] (a\x) at (\x*360/10:1cm) {};
\node[vertex] (b\x) at (\x*360/10:1.5cm) {};
\node[vertex] (c\x) at (\x*360/10:2cm) {};
}

\node (c10) at (-2.3,2.3) {\small $C_{10}$};
\draw [ ->] (c10) -- (120:2.05cm) {};
\draw [->] (c10) -- (125:1.55cm) {};
\draw [->] (c10) -- (130:1.05cm) {};
%\foreach \x/\y in {0/0,1/2,2/4,3/6,4/8} {
%\draw (v\x) -- (a\y);
%\draw (b\y) -- (c\y);
%}
%
%\foreach \x/\y in {0/1,1/3,2/5,3/7,4/9} {
%\draw (a\y)--(b\y);
%\draw (c\y)--(w\x);
%}

\end{tikzpicture}
\hspace{0.25cm}
\begin{tikzpicture}[scale=0.9]
\tikzstyle{vertex}=[circle, draw=black, thick, fill=black!15!white, inner sep = 0.4mm]
\draw (0,0) circle (0.5cm);
\draw (0,0) circle (1cm);
\draw (0,0) circle (1.5cm);
\draw (0,0) circle (2cm);
\draw (0,0) circle (2.5cm);
\foreach \x in {0,1,...,4} {
\node [vertex] (v\x) at (\x*360/5:0.5cm) {};
\node [vertex] (w\x) at (36+\x*360/5:2.5cm) {};
 }
\foreach \x in {0,1,...,9} {
\node[vertex] (a\x) at (\x*360/10:1cm) {};
\node[vertex] (b\x) at (\x*360/10:1.5cm) {};
\node[vertex] (c\x) at (\x*360/10:2cm) {};
}
\foreach \x/\y in {0/0,1/2,2/4,3/6,4/8} {
\draw (v\x) -- (a\y);
\draw (b\y) -- (c\y);
}

\foreach \x/\y in {0/1,1/3,2/5,3/7,4/9} {
\draw (a\y)--(b\y);
\draw (c\y)--(w\x);
}

\end{tikzpicture}
\end{center}
\caption{Constructing the layered graph $N(5,3)$}
\label{fig:x3circular}
\end{figure}

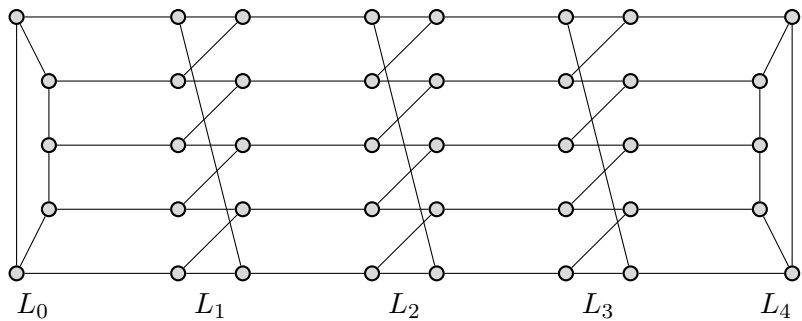
\begin{figure}
\begin{center}
\begin{tikzpicture}[scale=0.85]
\tikzstyle{vertex}=[circle, draw=black,  thick, fill=black!15!white, inner sep = 0.65mm]

\tikzstyle{edge0}=[black]
\tikzstyle{edge1}=[lightgray]
\tikzstyle{edge2}=[ultra thick, red]

\foreach \x in {0,1,...,4} {

\node [vertex] (b\x) at (2,\x) {};
\node [vertex] (c\x) at (3,\x) {};
\node [vertex] (d\x) at (5,\x) {};
\node [vertex] (e\x) at (6,\x) {};
\node [vertex] (f\x) at (8,\x) {};
\node [vertex] (g\x) at (9,\x) {};
}
\foreach \x in {1,2,3} {
\node [vertex] (a\x) at (0,\x) {};
\node [vertex] (h\x) at (11,\x) {};
}
\node [vertex] (a0) at (-0.5,0) {};
\node [vertex] (a4) at (-0.5,4) {};

\node [vertex] (h0) at (11.5,0) {};
\node [vertex] (h4) at (11.5,4) {};

\draw [edge0] (a0)--(a1)--(a2)--(a3)--(a4)--(a0);
\draw [edge0] (h0)--(h1)--(h2)--(h3)--(h4)--(h0);

%\node at (2.5,-0.75) { $ \underbrace{\hspace{1cm}}_{\text{Layer 1}} $};
%\node at (5.5,-0.75) { $ \underbrace{\hspace{1cm}}_{\text{Layer 2}} $};
%\node at (8.5,-0.75) { $ \underbrace{\hspace{1cm}}_{\text{Layer 3}} $};

\foreach \x in {0,1,...,4} {
\draw [edge0]  (a\x)--(b\x);
\draw [edge0]  (b\x)--(c\x);
\draw [edge0]  (c\x)--(d\x);
\draw [edge0]  (d\x)--(e\x);
\draw [edge0]  (e\x)--(f\x);
\draw [edge0]  (f\x)--(g\x);
\draw [edge0]  (g\x)--(h\x);
}

\foreach \x/\y in {0/1, 1/2, 2/3, 3/4, 4/0} {
\draw [edge0]  (b\x)--(c\y);
\draw [edge0]  (d\x)--(e\y);
\draw [edge0]  (f\x)--(g\y);
}

\node at (-0.25,-0.5) {$L_0$};
\node at (2.5,-0.5) {$L_1$};
\node at (5.5,-0.5) {$L_2$};
\node at (8.5,-0.5) {$L_3$};
\node at (11.25,-0.5) {$L_4$};

\end{tikzpicture}
\end{center}
\caption{The graph $N(5,3)$ redrawn}
\label{fig:x3linear}
\end{figure}

\begin{figure}
\begin{center}
\begin{tikzpicture}[yscale=1]
\tikzstyle{vertex}=[circle, draw=black, thick, inner sep = 0.1mm, minimum size = 5mm]
\foreach \x in {0,1,...,4} {
\node (e\x) at (0.25,\x) {\small $e_\x$};
\node [vertex] (v\x) at (2,\x) {\small $v_\x$};
\node [vertex] (w\x) at (4,\x) {\small $w_\x$};
\node (f\x) at (5.75,\x) {$f_\x$};
\draw [thick] (e\x)--(v\x)--(w\x)--(f\x);
}
\draw [thick] (v0)--(w1);
\draw  [thick](v1)--(w2);
\draw  [thick] (v2)--(w3);
\draw  [thick] (v3)--(w4);
\draw  [thick] (v4)--(w0);
\end{tikzpicture}
\end{center}
\caption{One of the $10$-vertex layers}
\label{fig:closeup}
\end{figure}
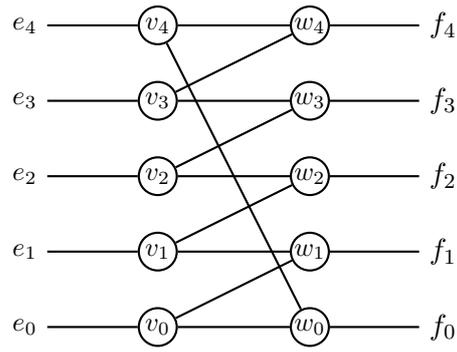

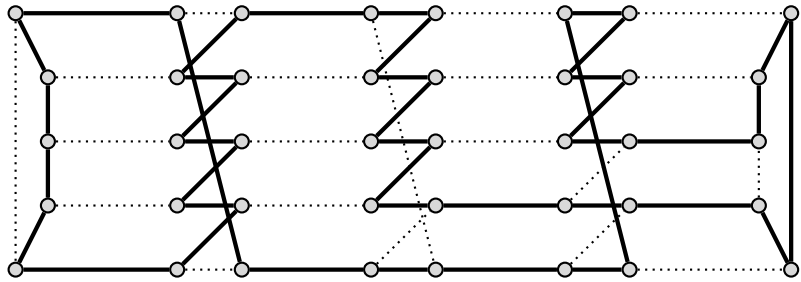
\begin{figure}
\begin{center}
\begin{tikzpicture}[scale=0.85]
\tikzstyle{vertex}=[circle, draw=black, thick, fill=black!15!white, inner sep = 0.65mm]
\foreach \x in {0,1,...,4} {

\node [vertex] (b\x) at (2,\x) {};
\node [vertex] (c\x) at (3,\x) {};
\node [vertex] (d\x) at (5,\x) {};
\node [vertex] (e\x) at (6,\x) {};
\node [vertex] (f\x) at (8,\x) {};
\node [vertex] (g\x) at (9,\x) {};
}
\foreach \x in {1,2,3} {
\node [vertex] (a\x) at (0,\x) {};
\node [vertex] (h\x) at (11,\x) {};
}
\node [vertex] (a0) at (-0.5,0) {};
\node [vertex] (a4) at (-0.5,4) {};

\node [vertex] (h0) at (11.5,0) {};
\node [vertex] (h4) at (11.5,4) {};

\draw [thick, dotted] (a0)--(a1)--(a2)--(a3)--(a4)--(a0);
\draw [thick, dotted] (h0)--(h1)--(h2)--(h3)--(h4)--(h0);

%\node at (2.5,-0.75) { $ \underbrace{\hspace{1cm}}_{\text{Layer 1}} $};
%\node at (5.5,-0.75) { $ \underbrace{\hspace{1cm}}_{\text{Layer 2}} $};
%\node at (8.5,-0.75) { $ \underbrace{\hspace{1cm}}_{\text{Layer 3}} $};

\foreach \x in {0,1,...,4} {
\draw [thick, dotted] (a\x)--(b\x);
\draw [thick, dotted] (b\x)--(c\x);
\draw [thick, dotted] (c\x)--(d\x);
\draw [thick, dotted] (d\x)--(e\x);
\draw [thick, dotted] (e\x)--(f\x);
\draw [thick, dotted] (f\x)--(g\x);
\draw [thick, dotted] (g\x)--(h\x);
}

\foreach \x/\y in {0/1, 1/2, 2/3, 3/4, 4/0} {
\draw [thick, dotted] (b\x)--(c\y);
\draw [thick, dotted] (d\x)--(e\y);
\draw [thick, dotted] (f\x)--(g\y);
}

\draw[ultra thick] (a0)--(a1);
\draw[ultra thick] (a1)--(a2);
\draw[ultra thick] (a2)--(a3);
\draw[ultra thick] (a3)--(a4);
\draw[ultra thick] (a4)--(b4);
\draw[ultra thick] (b4)--(c0);
\draw[ultra thick] (c0)--(d0);
\draw[ultra thick] (d0)--(e0);
\draw[ultra thick] (e0)--(f0);
\draw[ultra thick] (f0)--(g0);
\draw[ultra thick] (g0)--(f4);
\draw[ultra thick] (f4)--(g4);
\draw[ultra thick] (g4)--(f3);
\draw[ultra thick] (f3)--(g3);
\draw[ultra thick] (g3)--(f2);
\draw[ultra thick] (f2)--(g2);
\draw[ultra thick] (g2)--(h2);
\draw[ultra thick] (h2)--(h3);
\draw[ultra thick] (h3)--(h4);
\draw[ultra thick] (h4)--(h0);
\draw[ultra thick] (h0)--(h1);
\draw[ultra thick] (h1)--(g1);
\draw[ultra thick] (g1)--(f1);
\draw[ultra thick] (f1)--(e1);
\draw[ultra thick] (e1)--(d1);
\draw[ultra thick] (d1)--(e2);
\draw[ultra thick] (e2)--(d2);
\draw[ultra thick] (d2)--(e3);
\draw[ultra thick] (e3)--(d3);
\draw[ultra thick] (d3)--(e4);
\draw[ultra thick] (e4)--(d4);
\draw[ultra thick] (d4)--(c4);
\draw[ultra thick] (c4)--(b3);
\draw[ultra thick] (b3)--(c3);
\draw[ultra thick] (c3)--(b2);
\draw[ultra thick] (b2)--(c2);
\draw[ultra thick] (c2)--(b1);
\draw[ultra thick] (b1)--(c1);
\draw[ultra thick] (c1)--(b0);
\draw[ultra thick] (b0)--(a0);

\end{tikzpicture}
\end{center}
\caption{A Type 2 Hamilton cycle in the graph $N(5,3)$}
\label{fig:type2}
\end{figure}

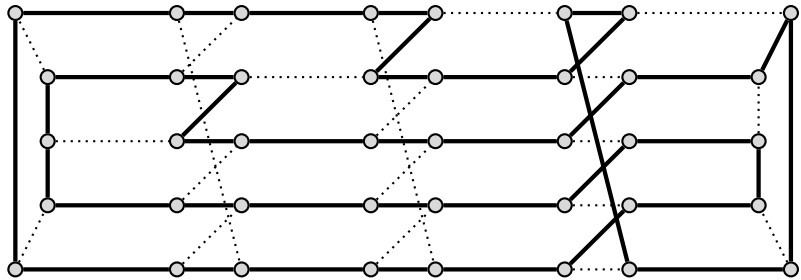
\begin{figure}
\begin{center}
\begin{tikzpicture}[scale=0.85]
\tikzstyle{vertex}=[circle, draw=black, thick, fill=black!15!white, inner sep = 0.65mm]
\tikzstyle{nonhamedge}=[thick, dotted]
\tikzstyle{hamedge}=[ultra thick]
\foreach \x in {0,1,...,4} {

\node [vertex] (b\x) at (2,\x) {};
\node [vertex] (c\x) at (3,\x) {};
\node [vertex] (d\x) at (5,\x) {};
\node [vertex] (e\x) at (6,\x) {};
\node [vertex] (f\x) at (8,\x) {};
\node [vertex] (g\x) at (9,\x) {};
}
\foreach \x in {1,2,3} {
\node [vertex] (a\x) at (0,\x) {};
\node [vertex] (h\x) at (11,\x) {};
}
\node [vertex] (a0) at (-0.5,0) {};
\node [vertex] (a4) at (-0.5,4) {};

\node [vertex] (h0) at (11.5,0) {};
\node [vertex] (h4) at (11.5,4) {};

\draw [nonhamedge] (a0)--(a1)--(a2)--(a3)--(a4)--(a0);
\draw [nonhamedge] (h0)--(h1)--(h2)--(h3)--(h4)--(h0);

%\node at (2.5,-0.75) { $ \underbrace{\hspace{1cm}}_{\text{Layer 1}} $};
%\node at (5.5,-0.75) { $ \underbrace{\hspace{1cm}}_{\text{Layer 2}} $};
%\node at (8.5,-0.75) { $ \underbrace{\hspace{1cm}}_{\text{Layer 3}} $};

\foreach \x in {0,1,...,4} {
\draw [nonhamedge] (a\x)--(b\x);
\draw  [nonhamedge](b\x)--(c\x);
\draw  [nonhamedge](c\x)--(d\x);
\draw  [nonhamedge](d\x)--(e\x);
\draw  [nonhamedge](e\x)--(f\x);
\draw  [nonhamedge](f\x)--(g\x);
\draw  [nonhamedge] (g\x)--(h\x);
}

\foreach \x/\y in {0/1, 1/2, 2/3, 3/4, 4/0} {
\draw [nonhamedge] (b\x)--(c\y);
\draw [nonhamedge] (d\x)--(e\y);
\draw [nonhamedge] (f\x)--(g\y);
}

\fill [white,opacity=0] (7,-0.5) rectangle (12,5);

\draw[hamedge] (a0)--(a4);
\draw[hamedge] (a4)--(b4);
\draw[hamedge] (b4)--(c4);
\draw[hamedge] (c4)--(d4);
\draw[hamedge] (d4)--(e4);
\draw[hamedge] (e4)--(d3);
\draw[hamedge] (d3)--(e3);
\draw[hamedge] (e3)--(f3);
\draw[hamedge] (f3)--(g4);
\draw[hamedge] (g4)--(f4);
\draw[hamedge] (f4)--(g0);
\draw[hamedge] (g0)--(h0);
\draw[hamedge] (h0)--(h4);
\draw[hamedge] (h4)--(h3);
\draw[hamedge] (h3)--(g3);
\draw[hamedge] (g3)--(f2);
\draw[hamedge] (f2)--(e2);
\draw[hamedge] (e2)--(d2);
\draw[hamedge] (d2)--(c2);
\draw[hamedge] (c2)--(b2);
\draw[hamedge] (b2)--(c3);
\draw[hamedge] (c3)--(b3);
\draw[hamedge] (b3)--(a3);
\draw[hamedge] (a3)--(a2);
\draw[hamedge] (a2)--(a1);
\draw[hamedge] (a1)--(b1);
\draw[hamedge] (b1)--(c1);
\draw[hamedge] (c1)--(d1);
\draw[hamedge] (d1)--(e1);
\draw[hamedge] (e1)--(f1);
\draw[hamedge] (f1)--(g2);
\draw[hamedge] (g2)--(h2);
\draw[hamedge] (h2)--(h1);
\draw[hamedge] (h1)--(g1);
\draw[hamedge] (g1)--(f0);
\draw[hamedge] (f0)--(e0);
\draw[hamedge] (e0)--(d0);
\draw[hamedge] (d0)--(c0);
\draw[hamedge] (c0)--(b0);
\draw[hamedge] (b0)--(a0);

\end{tikzpicture}
\end{center}
\caption{A Type 4 Hamilton cycle in the graph $N(5,3)$}
\label{fig:type4}
\end{figure}

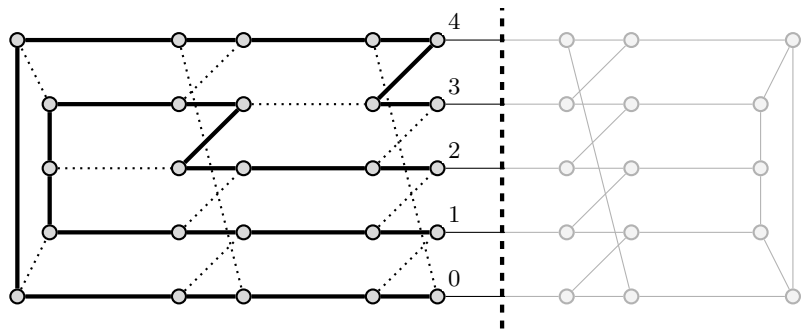
\begin{figure}
\begin{center}
\begin{tikzpicture}[scale=0.85]
\tikzstyle{vertex}=[circle, draw=black, thick, fill=black!15!white, inner sep = 0.65mm]
\tikzstyle{nonhamedge}=[thick, dotted]
\tikzstyle{hamedge}=[ultra thick]
\foreach \x in {0,1,...,4} {

\node [vertex] (b\x) at (2,\x) {};
\node [vertex] (c\x) at (3,\x) {};
\node [vertex] (d\x) at (5,\x) {};
\node [vertex] (e\x) at (6,\x) {};
\node [vertex] (f\x) at (8,\x) {};
\node [vertex] (g\x) at (9,\x) {};
}
\foreach \x in {1,2,3} {
\node [vertex] (a\x) at (0,\x) {};
\node [vertex] (h\x) at (11,\x) {};
}
\node [vertex] (a0) at (-0.5,0) {};
\node [vertex] (a4) at (-0.5,4) {};

\node [vertex] (h0) at (11.5,0) {};
\node [vertex] (h4) at (11.5,4) {};

\draw [nonhamedge] (a0)--(a1)--(a2)--(a3)--(a4)--(a0);
\draw (h0)--(h1)--(h2)--(h3)--(h4)--(h0);

%\node at (2.5,-0.75) { $ \underbrace{\hspace{1cm}}_{\text{Layer 1}} $};
%\node at (5.5,-0.75) { $ \underbrace{\hspace{1cm}}_{\text{Layer 2}} $};
%\node at (8.5,-0.75) { $ \underbrace{\hspace{1cm}}_{\text{Layer 3}} $};

\foreach \x in {0,1,...,4} {
\draw [nonhamedge] (a\x)--(b\x);
\draw [nonhamedge] (b\x)--(c\x);
\draw [nonhamedge] (c\x)--(d\x);
\draw [nonhamedge] (d\x)--(e\x);
\draw (e\x)--(f\x);
\draw (f\x)--(g\x);
\draw (g\x)--(h\x);
}

\foreach \x/\y in {0/1, 1/2, 2/3, 3/4, 4/0} {
\draw [nonhamedge] (b\x)--(c\y);
\draw [nonhamedge] (d\x)--(e\y);
\draw (f\x)--(g\y);
}

\draw[hamedge] (a0)--(a4);
\draw[hamedge] (a4)--(b4);
\draw[hamedge] (b4)--(c4);
\draw[hamedge] (c4)--(d4);
\draw[hamedge] (d4)--(e4);
\draw[hamedge] (e4)--(d3);
\draw[hamedge] (d3)--(e3);
%\draw[ultra thick, red] (e3)--(f3);
%\draw[ultra thick, red] (f3)--(g4);
%\draw[ultra thick, red] (g4)--(f4);
%\draw[ultra thick, red] (f4)--(g0);
%\draw[ultra thick, red] (g0)--(h0);
%\draw[ultra thick, red] (h0)--(h4);
%\draw[ultra thick, red] (h4)--(h3);
%\draw[ultra thick, red] (h3)--(g3);
%\draw[ultra thick, red] (g3)--(f2);
%\draw[ultra thick, red] (f2)--(e2);
\draw[hamedge] (e2)--(d2);
\draw[hamedge] (d2)--(c2);
\draw[hamedge] (c2)--(b2);
\draw[hamedge] (b2)--(c3);
\draw[hamedge] (c3)--(b3);
\draw[hamedge] (b3)--(a3);
\draw[hamedge] (a3)--(a2);
\draw[hamedge] (a2)--(a1);
\draw[hamedge] (a1)--(b1);
\draw[hamedge] (b1)--(c1);
\draw[hamedge] (c1)--(d1);
\draw[hamedge] (d1)--(e1);
%\draw[ultra thick, red] (e1)--(f1);
%\draw[ultra thick, red] (f1)--(g2);
%\draw[ultra thick, red] (g2)--(h2);
%\draw[ultra thick, red] (h2)--(h1);
%\draw[ultra thick, red] (h1)--(g1);
%\draw[ultra thick, red] (g1)--(f0);
%\draw[ultra thick, red] (f0)--(e0);
\draw[hamedge] (e0)--(d0);
\draw[hamedge] (d0)--(c0);
\draw[hamedge] (c0)--(b0);
\draw[hamedge] (b0)--(a0);

\draw [ultra thick, dashed] (7,-0.5)--(7,4.5);

\fill [white,opacity=0.7] (7.05,-0.5) rectangle (12,5);

%\node at (e1) {$*$};

\node [above right, fill=white, outer sep = 2pt, inner sep = 2pt] at (6,0) {\footnotesize $0$};
\node [above right, fill=white, outer sep = 2pt, inner sep = 2pt] at (6,1) {\footnotesize $1$};
\node [above right, fill=white, outer sep = 2pt, inner sep = 2pt] at (6,2) {\footnotesize $2$};
\node [above right, fill=white, outer sep = 2pt, inner sep = 2pt] at (6,3) {\footnotesize $3$};
\node [above right, fill=white, outer sep = 2pt, inner sep = 2pt] at (6,4) {\footnotesize $4$};

\end{tikzpicture}
\end{center}
\caption{The partial Hamilton cycle $H(0,2)$ has terminal partition $\{03|12|4\}$}
\label{fig:partialhamcyc}
\end{figure}

\section{Wider nanotubes}

It is easy to see how the definition of the nanotubes $N(5,k)$ can be generalized to the family $N(w,k)$ for any fixed $w > 5$ (here $w$ stands for ``width'') --- replace the $5$-cycles with $w$-cycles and the $10$-cycles with $2w$-cycles, and connect them using the same alternating in-out pattern as for the family $N(5,k)$. Although these graphs are not even fullerenes for $w > 6$ we will still call them nanotubes. In this section we will outline how the method above can be generalized to enumerate Hamilton cycles in the family $N(w,k)$ for any moderate fixed with $k$. The resulting expression involves a matrix product, and while the exact value for the number of Hamilton cycles can easily be evaluated for specific values of $k$, there may be no exact formula. In this latter case however, it can be possible to extract the asymptotic rate of growth using the eigenstructure of the transfer matrix.

As before, any Hamilton cycle in the graph $N(w,k)$ will use a fixed even number of the edges between consecutive layers, and we separately count the number of Hamilton cycles of Type $2$, Type $4$ and so on. Suppose then that we are trying to count the number of Hamilton cycles of Type $2c$ for some $c$ such that $2c \leqslant w$. 

Let $\mathcal{P}$ denote the set of non-crossing partitions of $\{0, 1, \ldots, w-1\}$ with exactly $c$ cells of size $2$ and $w-2c$ singletons, and suppose that  
\[
\mathcal{P} = \mathcal{O}_1 \cup \mathcal{O}_2 \cup \cdots \cup \mathcal{O}_\ell
\]
is the partition of $\mathcal{P}$ into the orbits of the group generated by the rotation  $i \mapsto i+1 \pmod{w}$.  

Now form an $\ell \times \ell$ matrix $M$, which we will call the \emph{transfer matrix}, with rows and columns indexed by $\{1, 2, \ldots, \ell\}$. For each $i$ such that $1 \leqslant i \leqslant \ell$, pick a particular terminal partition $\pi$ from $\mathcal{O}_i$ and let the $ij$-entry of $M$ be the \emph{number of tiles} that transfers $\pi$ to a terminal partition in the orbit $\mathcal{O}_j$. By the symmetry of the situation, $M_{ij}$ is independent of the actual choice of $\pi$.

In \cref{cub5conmany} we implicitly calculated the entries of the transfer matrix, but due to their particularly simple form, we did not need to explicitly put them into matrix form. Had we done so, we would have obtained the very simple transfer matrix 
\[
M = \left[
\begin{array}{cc}
0&3\\
4&0
\end{array}
\right].
\]

\newcommand{\vs}{v_\mathrm{s}}
\newcommand{\vf}{v_\mathrm{f}}

Now define two vectors, $\vs$ and $\vf$ (for ``start'' and ``finish'') indexed by $\{1,2,\ldots,\ell\}$, where the $i$-th  entry of $\vs$ is the number of starting tiles (choices for $H(0)$) with terminal partition in $\mathcal{O}_i$, while the $i$-th entry of $\vf$ is the number of ways in which a partial Hamilton cycle with terminal partition in $\mathcal{O}_i$ can be \emph{completed} to a Hamilton cycle by the addition of a finishing tile.

For the $w=5$ and $2c=4$ discussed in \cref{cub5conmany} we have
\[
\vs = 
\left[
\begin{array}{c}
0\\
5
\end{array}
\right]
\qquad
\vf = 
\left[
\begin{array}{c}
1\\
0
\end{array}
\right].
\]

\begin{theorem}
If $M$ is the transfer matrix as defined above, and $\vs$ and $\vf$ the starting and finishing vectors, then the number of
Hamilton cycles in $N(w,k)$ is given by the single entry of the $1 \times 1$ matrix
\[
\vs^T M^k \vf.
\]
\end{theorem}
\begin{proof}
We claim that for each $r \geqslant 0$, the $j$-th entry of the row vector $\vs^T M^r$ is the number of partial Hamilton cycles $H(0,r)$ whose terminal
partition lies in $\mathcal{O}_j$. This is true when $r=0$, so we will show that this assertion remains true if the vector $\vs^T M^r$ is multiplied by $M$. 
The $j$-th coordinate of  $\vs^T M^{r+1}$ is given by
\[
(\vs^T M^{r+1})_j = \sum_{1 \leqslant i \leqslant \ell} (\vs^T M^{r})_i M_{ij}.
\]
This is the sum over all orbit indices $i$ of the product of the number of partial Hamilton cycles $H(0,r)$ with terminal partition in $\mathcal{O}_i$ by the number of ways in which a terminal partition in $\mathcal{O}_i$ can be transferred to a terminal partition in $\mathcal{O}_j$. This is exactly the total number of partial Hamilton cycles $H(0,r+1)$ with terminal partition in $\mathcal{O}_j$, as required. 

Similarly, the final step -- multiplication by $\vf$ -- is the sum over all orbit indices $i$ of the number of partial Hamilton cycles with terminal partition in $\mathcal{O}_i$ multiplied by the number of ways in which such a partial Hamilton cycle can be completed to a Hamilton cycle. 
\end{proof}

Again we can recover the results of \cref{cub5conmany} in this context as
\[
M^k =
 \left[
\begin{array}{cc}
12^{k/2} & 0\\
0 & 12^{k/2}
\end{array}
\right] 
\qquad \text{or }
M^k =
\left[
\begin{array}{cc}
0 & 3 \cdot 12^{(k-1)/2}\\
4 \cdot 12^{(k-1)/2} & 0 
\end{array}
\right] 
\]
for $k$ even and $k$ odd, respectively, and so for $k$ odd we get
\[
\vs^T M^k \vf
=
\left[
\begin{array}{cc}
0 & 5\\
\end{array}
\right]
\left[
\begin{array}{cc}
0 & 3 \cdot 12^{(k-1)/2}\\
4 \cdot 12^{(k-1)/2} & 0 
\end{array}
\right]
\left[
\begin{array}{c}
1\\
0
\end{array}
\right]
= 
\left[
20 \cdot 12^{(k-1)/2}
\right].
\]

For any particular choice of $w$ and $c$, the computation of the transfer matrix and the start- and finish-vectors can be totally automated, yielding a formula  involving a matrix power. For a computer algebra system, finding a specific power of an integer matrix is relatively easy, even if both the matrix and the power are quite large. However a list of explicit values for the numbers of Hamilton cycles in a selection of nanotubes of varying lengths (but always fixed width $w$) does not immediately reveal the asymptotic behaviour of these numbers.  In these cases, it is necessary to examine the eigenstructure of transfer matrix in order to extract information about the asymptotic growth of the number of Hamilton cycles as the length of the nanotube increases.

\begin{figure}
\begin{center}
\begin{tikzpicture}[scale=0.6]
\tikzstyle{vertex}=[circle, draw=black, thick, fill = black!15!white, inner sep = 0.65mm]
\tikzstyle{hamedge}=[ultra thick]
\foreach \x in {1,2,3} {
\node [vertex] (a\x) at (0,\x) {};
}
\node [vertex] (a0) at (-0.5,0) {};
\node [vertex] (a4) at (-0.5,4) {};
\draw [dotted] (a0)--(a1)--(a2)--(a3)--(a4)--(a0);
\draw  [dotted](a0)--(1,0);
\draw  [dotted](a1)--(1,1);
\draw  [dotted](a2)--(1,2);
\draw  [dotted](a3)--(1,3);
\draw [dotted] (a4)--(1,4);
\draw [hamedge]  (0.75,3)--(a3)--(a4)--(a0)--(a1)--(a2)--(0.75,2);
\node [right]at (1,2) {\small $2$};
\node [right]at (1,3) {\small $3$};
\pgftransformxshift{5cm}
\tile
\draw [hamedge] (r3)--(w3)--(v3)--(l3);
\draw [hamedge] (r4)--(w4)--(v4)--(w0)--(v0)--(w1)--(v1)--(w2)--(v2)--(l2);
\node at (ll3) {\small $3$};
\node at (ll2) {\small $2$};
\node at (rr3) {\small $3$};
\node at (rr4) {\small $4$};
\pgftransformxshift{7cm}
\tile
\draw[hamedge] (l2)--(v2)--(w3)--(r3);
\draw[hamedge] (l3)--(v3) -- (w4)--(v4) -- (w0)--(v0)--(w1)--(v1)--(w2) -- (r2);
\node at (ll3) {\small $3$};
\node at (ll2) {\small $2$};
\node at (rr2) {\small $2$};
\node at (rr3) {\small $3$};
\end{tikzpicture}
\end{center}
\caption{Tile representatives for a Type 2 Hamilton cycle}
\label{fig:type2tiles}
\end{figure}
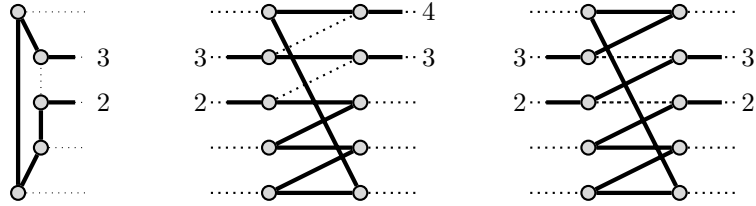

\begin{figure}
\begin{center}
\begin{tikzpicture}[scale=0.7]
\tikzstyle{vertex}=[circle, draw=black, thick, fill=black!15!white, inner sep = 0.65mm]
\tikzstyle{hamedge}=[ultra thick]
\tikzstyle{nothamedge}=[thick, dotted]
\foreach \x in {0,1,...,4} {

\node [vertex] (b\x) at (2,\x) {};
\node [vertex] (c\x) at (3,\x) {};
%\node [vertex] (d\x) at (5,\x) {};
%\node [vertex] (e\x) at (6,\x) {};
%\node [vertex] (f\x) at (8,\x) {};
%\node [vertex] (g\x) at (9,\x) {};
}
\foreach \x in {1,2,3} {
\node [vertex] (a\x) at (0,\x) {};
%\node [vertex] (h\x) at (11,\x) {};
}
\node [vertex] (a0) at (-0.5,0) {};
\node [vertex] (a4) at (-0.5,4) {};

%\node [vertex] (h0) at (11.5,0) {};
%\node [vertex] (h4) at (11.5,4) {};

\draw [nothamedge] (a0)--(a1)--(a2)--(a3)--(a4)--(a0);
%\draw (h0)--(h1)--(h2)--(h3)--(h4)--(h0);

%\node at (2.5,-0.75) { $ \underbrace{\hspace{1cm}}_{\text{Layer 1}} $};
%\node at (5.5,-0.75) { $ \underbrace{\hspace{1cm}}_{\text{Layer 2}} $};
%\node at (8.5,-0.75) { $ \underbrace{\hspace{1cm}}_{\text{Layer 3}} $};

\foreach \x in {0,1,...,4} {
\draw [nothamedge] (a\x)--(b\x);
\draw [nothamedge] (b\x)--(c\x);
%\draw (c\x)--(d\x);
%\draw (d\x)--(e\x);
%\draw (e\x)--(f\x);
%\draw (f\x)--(g\x);
%\draw (g\x)--(h\x);
}

\foreach \x/\y in {0/1, 1/2, 2/3, 3/4, 4/0} {
\draw [nothamedge] (b\x)--(c\y);
%\draw (d\x)--(e\y);
%\draw (f\x)--(g\y);
}

\draw[hamedge] (a0)--(a4);
\draw[hamedge] (a4)--(b4);
\draw[hamedge] (b4)--(c4);
%\draw[ultra thick, red] (c4)--(d4);
%\draw[ultra thick, red] (d4)--(e4);
%\draw[ultra thick, red] (e4)--(d3);
%\draw[ultra thick, red] (d3)--(e3);
%\draw[ultra thick, red] (e3)--(f3);
%\draw[ultra thick, red] (f3)--(g4);
%\draw[ultra thick, red] (g4)--(f4);
%\draw[ultra thick, red] (f4)--(g0);
%\draw[ultra thick, red] (g0)--(h0);
%\draw[ultra thick, red] (h0)--(h4);
%\draw[ultra thick, red] (h4)--(h3);
%\draw[ultra thick, red] (h3)--(g3);
%\draw[ultra thick, red] (g3)--(f2);
%\draw[ultra thick, red] (f2)--(e2);
%\draw[ultra thick, red] (e2)--(d2);
%\draw[ultra thick, red] (d2)--(c2);
\draw[hamedge] (c2)--(b2);
\draw[hamedge] (b2)--(c3);
\draw[hamedge] (c3)--(b3);
\draw[hamedge] (b3)--(a3);
\draw[hamedge] (a3)--(a2);
\draw[hamedge] (a2)--(a1);
\draw[hamedge] (a1)--(b1);
\draw[hamedge] (b1)--(c1);
%\draw[ultra thick, red] (c1)--(d1);
%\draw[ultra thick, red] (d1)--(e1);
%\draw[ultra thick, red] (e1)--(f1);
%\draw[ultra thick, red] (f1)--(g2);
%\draw[ultra thick, red] (g2)--(h2);
%\draw[ultra thick, red] (h2)--(h1);
%\draw[ultra thick, red] (h1)--(g1);
%\draw[ultra thick, red] (g1)--(f0);
%\draw[ultra thick, red] (f0)--(e0);
%\draw[ultra thick, red] (e0)--(d0);
%\draw[ultra thick, red] (d0)--(c0);
\draw[hamedge] (c0)--(b0);
\draw[hamedge] (b0)--(a0);

%\draw [ultra thick, dashed] (4,-0.25)--(4,4.25);
%\draw [<->] (3.25,1)--(3.25,2);
%\draw [<->] (3.35,0)--(3.35,4);

\draw [<->] (3.25,1)--(3.5,1)--(3.5,2)--(3.25,2);
\draw [<->] (3.25,0)--(3.75,0)--(3.75,4)--(3.25,4);

\end{tikzpicture}
\hspace{1cm}
\begin{tikzpicture}[scale=0.7]
\tikzstyle{vertex}=[circle, draw=black, thick, fill=black!15!white, inner sep = 0.65mm]
\tikzstyle{hamedge}=[ultra thick]
\tikzstyle{nothamedge}=[thick, dotted]
\foreach \x in {0,1,...,4} {

\node [vertex] (b\x) at (2,\x) {};
\node [vertex] (c\x) at (3,\x) {};
\node [vertex] (d\x) at (5,\x) {};
\node [vertex] (e\x) at (6,\x) {};
%\node [vertex] (f\x) at (8,\x) {};
%\node [vertex] (g\x) at (9,\x) {};
}
\foreach \x in {1,2,3} {
\node [vertex] (a\x) at (0,\x) {};
%\node [vertex] (h\x) at (11,\x) {};
}
\node [vertex] (a0) at (-0.5,0) {};
\node [vertex] (a4) at (-0.5,4) {};

%\node [vertex] (h0) at (11.5,0) {};
%\node [vertex] (h4) at (11.5,4) {};

\draw [nothamedge] (a0)--(a1)--(a2)--(a3)--(a4)--(a0);
%\draw (h0)--(h1)--(h2)--(h3)--(h4)--(h0);

%\node at (2.5,-0.75) { $ \underbrace{\hspace{1cm}}_{\text{Layer 1}} $};
%\node at (5.5,-0.75) { $ \underbrace{\hspace{1cm}}_{\text{Layer 2}} $};
%\node at (8.5,-0.75) { $ \underbrace{\hspace{1cm}}_{\text{Layer 3}} $};

\foreach \x in {0,1,...,4} {
\draw [nothamedge](a\x)--(b\x);
\draw [nothamedge](b\x)--(c\x);
\draw [nothamedge](c\x)--(d\x);
\draw [nothamedge](d\x)--(e\x);
%\draw (e\x)--(f\x);
%\draw (f\x)--(g\x);
%\draw (g\x)--(h\x);
}

\foreach \x/\y in {0/1, 1/2, 2/3, 3/4, 4/0} {
\draw [nothamedge](b\x)--(c\y);
\draw [nothamedge](d\x)--(e\y);
%\draw (f\x)--(g\y);
}

\draw[hamedge] (a0)--(a4);
\draw[hamedge] (a4)--(b4);
\draw[hamedge] (b4)--(c4);
\draw[hamedge] (c4)--(d4);
\draw[hamedge] (d4)--(e4);
\draw[hamedge] (e4)--(d3);
\draw[hamedge] (d3)--(e3);
%\draw[ultra thick, red] (e3)--(f3);
%\draw[ultra thick, red] (f3)--(g4);
%\draw[ultra thick, red] (g4)--(f4);
%\draw[ultra thick, red] (f4)--(g0);
%\draw[ultra thick, red] (g0)--(h0);
%\draw[ultra thick, red] (h0)--(h4);
%\draw[ultra thick, red] (h4)--(h3);
%\draw[ultra thick, red] (h3)--(g3);
%\draw[ultra thick, red] (g3)--(f2);
%\draw[ultra thick, red] (f2)--(e2);
\draw[hamedge] (e2)--(d2);
\draw[hamedge] (d2)--(c2);
\draw[hamedge] (c2)--(b2);
\draw[hamedge] (b2)--(c3);
\draw[hamedge] (c3)--(b3);
\draw[hamedge] (b3)--(a3);
\draw[hamedge] (a3)--(a2);
\draw[hamedge] (a2)--(a1);
\draw[hamedge] (a1)--(b1);
\draw[hamedge] (b1)--(c1);
\draw[hamedge] (c1)--(d1);
\draw[hamedge] (d1)--(e1);
%\draw[ultra thick, red] (e1)--(f1);
%\draw[ultra thick, red] (f1)--(g2);
%\draw[ultra thick, red] (g2)--(h2);
%\draw[ultra thick, red] (h2)--(h1);
%\draw[ultra thick, red] (h1)--(g1);
%\draw[ultra thick, red] (g1)--(f0);
%\draw[ultra thick, red] (f0)--(e0);
\draw[hamedge] (e0)--(d0);
\draw[hamedge] (d0)--(c0);
\draw[hamedge] (c0)--(b0);
\draw[hamedge] (b0)--(a0);

%\draw [ultra thick, dashed] (4,-0.25)--(4,4.25);
%\draw [<->] (3.25,1)--(3.25,2);
%\draw [<->] (3.35,0)--(3.35,4);

\draw [<->] (6.25,1)--(6.5,1)--(6.5,2)--(6.25,2);
\draw [<->] (6.25,0)--(6.75,0)--(6.75,3)--(6.25,3);

\end{tikzpicture}
\end{center}
\label{fig:extend}
\caption{Partial Hamilton cycle $H(0,1)$ extended to $H(0,2)$}
\label{fig:extending}
\end{figure}
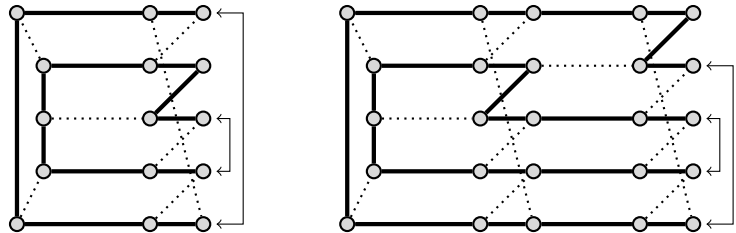

%\tikzset{%
%vert/.style = {circle, draw=black, thick, fill = red!15!white, inner sep = 0.65mm}
%}

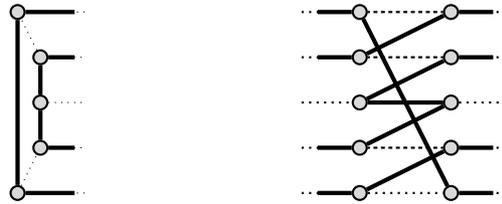
\begin{figure}
\begin{center}
\begin{tikzpicture}[scale=0.6]
\tikzstyle{vertex}=[circle, draw=black, thick, fill = black!15!white, inner sep = 0.65mm]
\tikzstyle{hamedge}=[ultra thick]
\foreach \x in {1,2,3} {
\node [vertex] (a\x) at (0,\x) {};
}
\node [vertex] (a0) at (-0.5,0) {};
\node [vertex] (a4) at (-0.5,4) {};
\draw [dotted] (a0)--(a1)--(a2)--(a3)--(a4)--(a0);
\draw [dotted] (a0)--(1,0);
\draw [dotted] (a1)--(1,1);
\draw [dotted] (a2)--(1,2);
\draw [dotted](a3)--(1,3);
\draw [dotted] (a4)--(1,4);

\draw [hamedge] (0.75,0)--(a0)--(a4)--(0.75,4);
\draw [hamedge] (0.75,1)--(a1)--(a2)--(a3)--(0.75,3);

\pgftransformxshift{7cm}

\tile

\draw [hamedge] (r4)--(w4)--(v3)--(l3);
\draw [hamedge] (r3)--(w3)--(v2)--(w2)--(v1)--(l1);
\draw [hamedge] (r1)--(w1)--(v0)--(l0);
\draw [hamedge] (r0)--(w0)--(v4)--(l4);

\end{tikzpicture}
\end{center}
\caption{A starting tile with terminal partition $\{04|13|2\}$ and a compatible second tile}
\label{fig:compat0}
\end{figure}

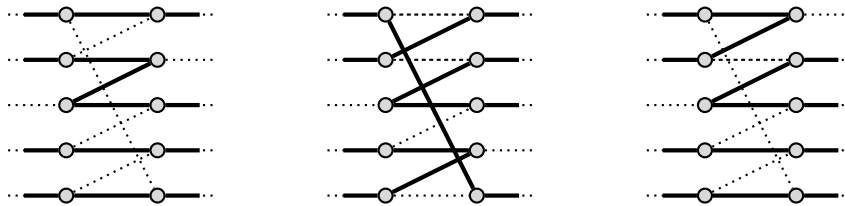
\begin{figure}
\begin{center}
\begin{tikzpicture}[scale=0.6]
\tikzstyle{vertex}=[circle, draw=black, thick, fill = black!15!white, inner sep = 0.65mm]
\tikzstyle{hamedge}=[ultra thick]

\tile 

\draw [hamedge] (r0)--(w0)--(v0)--(l0);
\draw [hamedge] (l1)--(v1)--(w1)--(r1);
\draw [hamedge] (l3)--(v3)--(w3)--(v2)--(w2)--(r2);
\draw [hamedge] (l4)--(v4)--(w4)--(r4);

% \draw (l0) [ultra thick, dotted] arc [start angle=270, end angle = 90, x radius =1, y radius=2];
% \draw (l1) [ultra thick, dotted] arc [start angle=270, end angle = 90, x radius = 0.5, y radius=1];

\pgftransformxshift{7cm}

\tile

% \draw (l0) [ultra thick, dotted] arc [start angle=270, end angle = 90, x radius =1, y radius=2];
% \draw (l1) [ultra thick, dotted] arc [start angle=270, end angle = 90, x radius = 0.5, y radius=1];

\draw [hamedge] (r4)--(w4)--(v3)--(l3);

\draw [hamedge] (l1)--(v1)--(w1)--(v0)--(l0);
\draw [hamedge] (r0)--(w0)--(v4)--(l4);
\draw [hamedge] (r3)--(w3)--(v2)--(w2)--(r2);

\pgftransformxshift{7cm}

\tile

% \draw (l0) [ultra thick, dotted] arc [start angle=270, end angle = 90, x radius =1, y radius=2];
% \draw (l1) [ultra thick, dotted] arc [start angle=270, end angle = 90, x radius = 0.5, y radius=1];

\draw [hamedge] (r0)--(w0)--(v0)--(l0);
\draw [hamedge] (r1)--(w1)--(v1)--(l1);
\draw [hamedge] (r3)--(w3)--(v2)--(w2)--(r2);
\draw [hamedge] (l4)--(v4)--(w4)--(v3)--(l3);

\end{tikzpicture}

\end{center}
\caption{Three (more) tiles compatible with terminal partition $\{04|13|2\}$.}
\label{fig:compat1}
\end{figure}

%%%%%%%%%%%%%%%%%%%%%%%%%%%%%%%%%%%%%%%%%%%%%%%%%%%%%%%%%%

\begin{figure}
\begin{center}
\begin{tikzpicture}[scale=0.6]
\tikzstyle{vertex}=[circle, draw=black, thick, fill = black!15!white, inner sep = 0.65mm]
\tikzstyle{hamedge}=[ultra thick]

\tile

%\draw (l1) [ultra thick, dotted] arc [start angle=90, end angle = 270, x radius = 0.5, y radius=0.5];
%\draw (l4) [ultra thick, dotted] arc [start angle=90, end angle = 270, x radius = 0.5, y radius=0.5];

\draw [hamedge] (l4)--(v4)--(w0)--(r0);
\draw [hamedge] (r3)--(w3)--(v2)--(w2)--(v1)--(l1);
\draw [hamedge] (l0)--(v0)--(w1)--(r1);
\draw [hamedge] (l3)--(v3)--(w4)--(r4);
%\draw (l3)--(v3)--(r4)--(r3);

\pgftransformxshift{7cm}

\tile

%\draw (l1) [ultra thick, red, dotted] arc [start angle=90, end angle = 270, x radius = 0.5, y radius=0.5];
%\draw (l4) [ultra thick, red, dotted] arc [start angle=90, end angle = 270, x radius = 0.5, y radius=0.5];

\draw [hamedge] (r4)--(w4)--(v3)--(l3);
\draw [hamedge] (l4)--(v4)--(w0)--(v0)--(l0);
\draw [hamedge] (l1)--(v1)--(w1)--(r1);
\draw [hamedge] (r3)--(w3)--(v2)--(w2)--(r2);
\pgftransformxshift{7cm}

\tile

%\draw (l1) [ultra thick, red, dotted] arc [start angle=90, end angle = 270, x radius = 0.5, y radius=0.5];
%\draw (l4) [ultra thick, red, dotted] arc [start angle=90, end angle = 270, x radius = 0.5, y radius=0.5];

\draw [hamedge] (r4)--(w4)--(v4)--(l4);
\draw [hamedge] (l3)--(v3)--(w3)--(v2)--(w2)--(r2);
\draw [hamedge] (l0)--(v0)--(w0)--(r0);
\draw [hamedge] (l1)--(v1)--(w1)--(r1);

\end{tikzpicture}
\end{center}
\caption{Internal tiles compatible with terminal partition $\{01|2|34\}$}
\label{fig:compat2}
\end{figure}
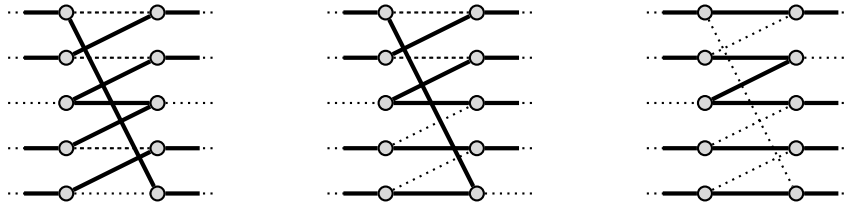

% \texorpdfstring{$5$}{5}

\section{The family \texorpdfstring{$N(6,k)$}{N(6,k)}}

In this section, we illustrate the transfer matrix method on a larger example, in this case finding the number of Type 4 Hamilton cycles in the infinite family of width-$6$ nanotubes $N(6,k)$.

There are $30$ non-crossing partitions of $\{0,1,2,3,4,5\}$ which fall into $6$ orbits under the rotation of order $6$ as follows:
\begin{align*}
\mathcal{O}_1 &= \langle \{ 0 | 1 | 2 3 | 4 5 \} \rangle ,\qquad
\mathcal{O}_2 = \langle \{ 0 | 1 | 2 5 | 3 4 \} \rangle ,\\
\mathcal{O}_3 &= \langle \{ 0 | 1 3 | 2 | 4 5 \} \rangle, \qquad
\mathcal{O}_4 = \langle \{ 0 | 1 5 | 2 | 3 4 \} \rangle ,\\
\mathcal{O}_5 &= \langle \{ 0 | 1 2 | 3 | 4 5 \} \rangle ,\qquad
\mathcal{O}_6 = \langle \{ 0 | 1 5 | 2 4 | 3 \} \rangle .
\end{align*}

Orbits $\mathcal{O}_5$ and $\mathcal{O}_6$ have size $3$ (because the rotation $i \mapsto i+3 \pmod 6$ fixes each partition from this orbit), while the remainder have size $6$.  The transfer matrix is 
\[
M = \left[
\begin{array}{cccccc}
0 & 2 & 0 & 0 &0 & 1\\
2 & 0 & 1 & 1 & 0 & 0\\
1 & 0 & 1 & 2 & 1 & 0\\
1 & 0 & 2 & 1 & 1 & 0\\
0 & 2 & 0 & 0 & 0 & 2\\
0 & 0 & 2 & 2 & 2 & 0\\
\end{array}
\right].
\]

For example, there are three tiles consistent with the terminal partition $\{0|1|23|45\}$, which is the representative of the orbit $\mathcal{O}_1$. These three tiles, illustrated in \cref{fig:sixtiles}, transfer $\{0|1|23|45\}$ to $\{03|1|2|45\}$, $\{02|1|35|4\}$ and $\{0|1|25|34\}$ respectively.  These three terminal partitions lie in $\mathcal{O}_2$, $\mathcal{O}_6$ and $\mathcal{O}_2$ respectively, and so the only non-zero entries in the first row of $M$ are $M_{12} = 2$ and   $M_{16} = 1$.

\begin{figure}
\begin{center}
\begin{tikzpicture}[scale=0.6]
\tikzstyle{vertex}=[circle, draw=black, thick, fill = black!15!white, inner sep = 0.65mm]
\sixtile
\draw [ultra thick, dotted] [bend left = 90] (l2) to (l3);
\draw [ultra thick, dotted] [bend left = 90] (l4) to (l5);
\draw [ultra thick] (ll2)--(v2);
\draw [ultra thick] (ll3)--(v3);
\draw [ultra thick] (ll4)--(v4);
\draw [ultra thick] (ll5)--(v5);
\draw [ultra thick] (w2)--(v1)--(w1)--(v0)--(w0);
\draw [ultra thick] (v2)--(w2);
\draw [ultra thick] (w0)--(r0);
\draw [ultra thick] (w3)--(r3);
\draw [ultra thick] (w4)--(r4);
\draw [ultra thick] (w5)--(r5);
\draw [ultra thick] (v3)--(w3);
\draw [ultra thick] (v4)--(w4);
\draw [ultra thick] (v5)--(w5);
\pgftransformxshift{7cm}
\sixtile
\draw [ultra thick, dotted] [bend left = 90] (l2) to (l3);
\draw [ultra thick, dotted] [bend left = 90] (l4) to (l5);
\draw [ultra thick] (ll2)--(v2);
\draw [ultra thick] (ll3)--(v3);
\draw [ultra thick] (ll4)--(v4);
\draw [ultra thick] (ll5)--(v5);
\draw [ultra thick] (w2)--(v1)--(w1)--(v0)--(w0);
\foreach \x in {0,2,3,5} {
  \draw [ultra thick] (w\x)--(r\x);
 }
\draw [ultra thick] (v4)--(w4)--(v3);
\draw [ultra thick] (v5)--(w5);
\draw [ultra thick] (v2)--(w3);
\pgftransformxshift{7cm}
\sixtile
\draw [ultra thick, dotted] [bend left = 90] (l2) to (l3);
\draw [ultra thick, dotted] [bend left = 90] (l4) to (l5);
\draw [ultra thick] (ll2)--(v2);
\draw [ultra thick] (ll3)--(v3);
\draw [ultra thick] (ll4)--(v4);
\draw [ultra thick] (ll5)--(v5);
\draw [ultra thick] (w2)--(v1)--(w1)--(v0)--(w0);
\foreach \x in {2,3,4,5} {
  \draw [ultra thick] (w\x)--(r\x);
 }
 \draw [ultra thick] (w0)--(v5);
 \draw [ultra thick] (v2)--(w3);
 \draw [ultra thick] (v3)--(w4);
 \draw [ultra thick] (v4)--(w5);
\end{tikzpicture}
\caption{Tiles consistent with $\{0|1|23|45\}$}
\label{fig:sixtiles}
\end{center}
\end{figure}
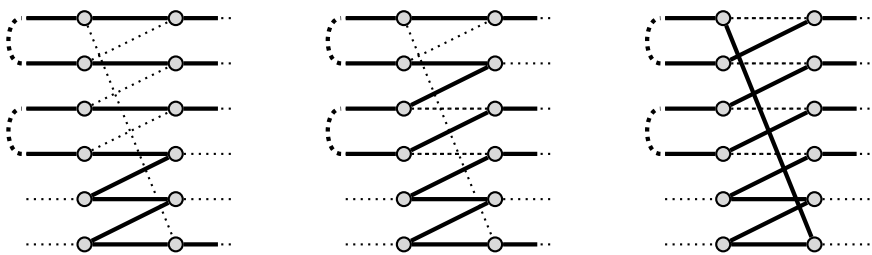

Checking all the possible initial endtiles (i.e. choices for $H(0)$) shows that there are only $9$ possible initial terminal partitions, namely any terminal partition in $\mathcal{O}_1$ and any in $\mathcal{O}_6$, and therefore
\[
\vs = \left[
\begin{array}{cccccc}
0&6&0&0&0&3
\end{array}
\right]^T.
\]
By symmetry, there are only the same $9$ possible end tiles for the final layer, but now we are asking how many of them can \emph{complete} a partial Hamilton cycle with a particular terminal partition. A straightforward case analysis on the $6$ possible terminal partitions shows that only terminal partitions in $\mathcal{O}_1$ and $\mathcal{O}_5$ can be completed to a Hamilton cycle, each in a unique fashion. So 
\[
\vf = \left[
\begin{array}{cccccc}
1&0&0&0&1&0
\end{array}
\right]^T.
\]

Now we wish to calculate, or at least compute the asymptotics for, the expression $\vs^T M^k \vf$. The matrix $M$ is diagonalizable and so there is a basis $\{v_1, v_2, \ldots, v_6\}$ for $\mathbb{R}^6$ consisting of eigenvectors of $M$. We can then find an expression for $\vf$ of the form
\[
\vf = \alpha_1 v_1 + \alpha_2 v_2 + \cdots + \alpha_6 v_6
\]
from which it follows that 
\[
M^k \vf = \alpha_1 \lambda_1^k v_1 + \alpha_2 \lambda_2^k v_2 + \cdots + \alpha_6 \lambda_6^k v_6
\]
and finally
\[
\vs^T M^k \vf =\left(\alpha_1 \vs^T v_1\right) \lambda_1^k+\left(\alpha_2 \vs^T v_2\right) \lambda_2^k+ \cdots + \left( \alpha_6  \vs^T v_6 \right) \lambda_6^k.
\]
  
This expression is dominated by the largest eigenvalue $\lambda_i$ for which $\alpha_i \vs^T v_i \ne 0$. In this case, using a computer algebra system, it is possible to compute exact expressions for the eigenvalues of $M$, the eigenvectors of $M$ and the coefficients $\alpha_i$, but with the caveat that the intermediate terms and final expressions are given in terms of roots of specific integer polynomials. In our particular case, the final expression for $\left(\alpha_1 \vs^T v_1\right) \lambda_1^k$ has the surprisingly simple form $A B^k$, where $A \approx 2.756982978$ is the largest real root of the polynomial  $7x^3 - 63x + 27$ and $B \approx 4.493959207$ is the largest root of the polynomial $x^3 - 4x^2 - 4x  + 8$.   The expression $AB^k$ is correct to within $1\%$ of the exact value (as given by the matrix equation) by $k = 5$ and within $0.01\%$ by $k=10$.

We conclude by asking whether the nanotubes of width $6$ beat the nanotubes of width $5$, which we can analyse just by comparing rate of growth for each family expressed as a function of the number of vertices. (Notice that we are also ignoring the Hamilton cycles of Type 2 and Type 6 in the width 6 case, but as these are dominated by the number of Type 4, we are justified in doing so.) Therefore this is a comparison between $12^{v/20}$ for the $N(5,k)$ family and and $B^{v/12}$ for the $N(6,k)$ family.  However \[
12^{1/20} \approx 1.132293625 < 1.133406661 \approx B^{1/12},
\]
and so the nanotubes of width $6$ will (eventually) beat the nanotubes of width $5$. The smallest number of vertices for which there exists a graph in each of the families is when $v=60$, but $N(5,5)$ with $3040$ Hamilton cycles has dramatically more than $N(6,4)$ with $1232$ Hamilton cycles (of which $1104$ are Type 4).

\section{Conclusion}

While there is a vast literature on the \emph{existence} of Hamilton cycles in graphs, especially cubic graphs, and even more especially planar cubic graphs, there is much less on the \emph{enumeration} of Hamilton cycles.   

One of the first results in this area was Schwenk's \cite{MR1007713} enumeration of the number of Hamilton cycles in the generalized Petersen graphs $P(m,2)$. His very first open question, now $30$ years old -- whether it is possible to determine the number of Hamilton cycles in $P(m,k)$ for $k \geqslant 3$ -- has still not been answered.

It seems that Chia \& Thomassen \cite{MR2951794} were the first to explicitly consider counting the \emph{rate of growth} of the maximum number of Hamilton cycles in various classes of cubic graphs. For a range of classes of cubic graphs, defined primarily by connectivity constraints they found promising families of graphs rich in Hamilton cycles, and then asked whether these were actually the best possible. Given that this was likely to be done without computer assistance, it is a testament to their ingenuity that it seems difficult to find families of graphs richer in Hamilton cycles. In particular, they found a family of $2$-connected cubic graphs with $\left(\sqrt[12]{10}\right)^v$ Hamilton cycles and a family of $3$-connected graphs with $(3/\sqrt{2}) \left(\sqrt[4]{2}\right)^v$ Hamilton cycles. Either a proof that these families are indeed extremal or an example with more Hamilton cycles would be very interesting.

Finally, we turn to the single graph on $56$ vertices that somehow beats the generalized Petersen graph $P(28,2)$, which is pictured in \cref{fig:full56}. A rapid examination of the faces illustrated in the figure reveal that all have size $5$ or $6$ and so the graph is again a fullerene. The graph has an automorphism group of order $6$, but it is not cyclic so we cannot obtain a drawing with $6$-fold rotational symmetry.  It would be interesting to see if this graph belongs to another family of fullerenes with more Hamilton cycles than the generalized Petersen graphs.

%Further computations, but now restricted to fullerenes only, show that there is indeed a $96$-vertex fullerene with more Hamilton cycles than $P(48,2)$, and so we finish with the following question.

%
%finding ``good'' families of graphs with  bounds or conjectured bounds in the form $\alpha b^v$ where $v$ is the number of vertices. They found 
%
%\begin{table}
%\begin{center}
%\begin{tabular}{cc}
%\toprule
%Connectivity & Best Family \\
%\midrule
%$2$-connected
%$3$-connected
%$
%
%
%
%\end{tabular}
%\end{center}
%\end{table}

%In particular, they observed that any cubic multigraph has at most $\left(\sqrt{2}\right)^v$ Hamilton cycles, so . Restricting to simple cubic graphs, they exhibited infinite families of graphs with $\left(\sqrt[3]{2}\right)^v$ and $\left(\sqrt[12]{10}\right)^v$ Hamilton cycles respectively. (Note that the bases for these powers are approximately $1.25992$ and $1.28209$ respectively.) These families of graphs have many $2$-vertex cuts, and so they next considered $3$-connected graphs, finding a family of graphs with $(3/\sqrt{2}) \left(\sqrt[4]{2}\right)^v$ --- the base for this power is approximately $1.18921$. 

\begin{figure}
\begin{center}
\begin{tikzpicture}[scale=3]
\tikzstyle{vertex}=[circle, fill=blue!25!white, draw=black, inner sep = 0.2mm]
\coordinate (c0) at (0.0728109501765787, 0.298219183993922);
\coordinate (c1) at (-0.00628131802869049, 0.238339981328841);
\coordinate (c2) at (0.144673703653233, 0.236630652342169);
\coordinate (c3) at (0.0800404649051932, 0.419686918310756);
\coordinate (c4) at (-0.101682918795864, 0.256127135992194);
\coordinate (c5) at (0.0100280145332141, 0.160673624000408);
\coordinate (c6) at (0.136147449939479, 0.174404144495944);
\coordinate (c7) at (0.225062710843642, 0.237268628536641);
\coordinate (c8) at (0.294044168473376, 0.530213204456887);
\coordinate (c9) at (-0.126733723934375, 0.430628366481458);
\coordinate (c10) at (-0.150070059264266, 0.327269539491978);
\coordinate (c11) at (-0.148697379094636, 0.202771887155763);
\coordinate (c12) at (-0.0485560187459463, 0.104349212662495);
\coordinate (c13) at (0.0849213803742791, 0.139331678009888);
\coordinate (c14) at (0.178847265790926, 0.147250103135776);
\coordinate (c15) at (0.228422388362756, 0.170247941992289);
\coordinate (c16) at (0.302092040514936, 0.304927291275465);
\coordinate (c17) at (0.500000000000000, 0.866025403784439);
\coordinate (c18) at (-0.310171577444053, 0.544928641641640);
\coordinate (c19) at (-0.221793535062559, 0.295053116002283);
\coordinate (c20) at (-0.211529537525625, 0.219757653855847);
\coordinate (c21) at (-0.132879680962419, 0.132430871619248);
\coordinate (c22) at (-0.0228163898086337, 0.0199431423678276);
\coordinate (c23) at (0.108588676650144, 0.0829172655333112);
\coordinate (c24) at (0.171971959070542, 0.0970982229190935);
\coordinate (c25) at (0.281357188453701, 0.126225094304450);
\coordinate (c26) at (0.387169242227789, 0.147300040832868);
\coordinate (c27) at (1.00000000000000, 0.000000000000000);
\coordinate (c28) at (-0.500000000000000, 0.866025403784439);
\coordinate (c29) at (-0.303781008397784, 0.338132154659024);
\coordinate (c30) at (-0.264097698419681, 0.161447958409494);
\coordinate (c31) at (-0.201385645046675, 0.0901715150394873);
\coordinate (c32) at (-0.0887658411855650, -0.0568416812299644);
\coordinate (c33) at (0.0688726905056100, 0.0123218956709524);
\coordinate (c34) at (0.228479934770557, 0.0611273000881937);
\coordinate (c35) at (0.578058497714732, 0.0107477369186893);
\coordinate (c36) at (0.500000000000000, -0.866025403784439);
\coordinate (c37) at (-1.00000000000000, 0.000000000000000);
\coordinate (c38) at (-0.379377912686741, 0.174414706333149);
\coordinate (c39) at (-0.207179555757926, -0.0233642849102808);
\coordinate (c40) at (-0.0363015779901353, -0.167103901147440);
\coordinate (c41) at (0.120845784675320, -0.0658947208882816);
\coordinate (c42) at (0.232110656787427, -0.0399414169589628);
\coordinate (c43) at (0.347006250916405, -0.115056830076801);
\coordinate (c44) at (0.245424730311183, -0.555678298398803);
\coordinate (c45) at (-0.500000000000000, -0.866025403784439);
\coordinate (c46) at (-0.570255031242760, 0.0236640059309278);
\coordinate (c47) at (-0.331387181041537, -0.103422688540365);
\coordinate (c48) at (-0.0816928995177637, -0.274405380835521);
\coordinate (c49) at (0.0615540067329227, -0.170064641376834);
\coordinate (c50) at (0.230849598247057, -0.315976810190128);
\coordinate (c51) at (0.00542459268649184, -0.485032681221842);
\coordinate (c52) at (-0.237100787812478, -0.553875290549341);
\coordinate (c53) at (-0.216726956123926, -0.310567786641742);
\coordinate (c54) at (0.00794983556077058, -0.345544454717381);
\coordinate (c55) at (0.100117813513584, -0.277195302094781);
%
%
%
%\uncover<1>{
%\fill [yellow!50!white] (c54)
%--(c55)
%--(c50)
%--(c44)
%--(c51)
%--(c54)
%;
%\fill [yellow!50!white] (c55)
%--(c54)
%--(c48)
%--(c40)
%--(c49)
%--(c55)
%;
%\fill [yellow!50!white] (c19)
%--(c29)
%--(c18)
%--(c9)
%--(c10)
%--(c19)
%;
%\fill [yellow!50!white] (c29)
%--(c19)
%--(c20)
%--(c30)
%--(c38)
%--(c29)
%;
%\fill [yellow!50!white] (c25)
%--(c34)
%--(c24)
%--(c14)
%--(c15)
%--(c25)
%;
%\fill [yellow!50!white] (c6)
%--(c2)
%--(c7)
%--(c15)
%--(c14)
%--(c6)
%;
%\fill [yellow!50!white] (c53)
%--(c48)
%--(c54)
%--(c51)
%--(c52)
%--(c53)
%;
%\fill [yellow!50!white] (c20)
%--(c19)
%--(c10)
%--(c4)
%--(c11)
%--(c20)
%;
%\fill [yellow!50!white] (c11)
%--(c21)
%--(c31)
%--(c30)
%--(c20)
%--(c11)
%;
%\fill [yellow!50!white] (c7)
%--(c16)
%--(c26)
%--(c25)
%--(c15)
%--(c7)
%;
%\fill [yellow!50!white] (c14)
%--(c24)
%--(c23)
%--(c13)
%--(c6)
%--(c14)
%;
%\fill [yellow!50!white] (c36)
%--(c45)
%--(c52)
%--(c51)
%--(c44)
%--(c36)
%;
%}

\node [vertex] (v0) at (c0) {};
\node [vertex] (v1) at (c1) {};
\node [vertex] (v2) at (c2) {};
\node [vertex] (v3) at (c3) {};
\node [vertex] (v4) at (c4) {};
\node [vertex] (v5) at (c5) {};
\node [vertex] (v6) at (c6) {};
\node [vertex] (v7) at (c7) {};
\node [vertex] (v8) at (c8) {};
\node [vertex] (v9) at (c9) {};
\node [vertex] (v10) at (c10) {};
\node [vertex] (v11) at (c11) {};
\node [vertex] (v12) at (c12) {};
\node [vertex] (v13) at (c13) {};
\node [vertex] (v14) at (c14) {};
\node [vertex] (v15) at (c15) {};
\node [vertex] (v16) at (c16) {};
\node [vertex] (v17) at (c17) {};
\node [vertex] (v18) at (c18) {};
\node [vertex] (v19) at (c19) {};
\node [vertex] (v20) at (c20) {};
\node [vertex] (v21) at (c21) {};
\node [vertex] (v22) at (c22) {};
\node [vertex] (v23) at (c23) {};
\node [vertex] (v24) at (c24) {};
\node [vertex] (v25) at (c25) {};
\node [vertex] (v26) at (c26) {};
\node [vertex] (v27) at (c27) {};
\node [vertex] (v28) at (c28) {};
\node [vertex] (v29) at (c29) {};
\node [vertex] (v30) at (c30) {};
\node [vertex] (v31) at (c31) {};
\node [vertex] (v32) at (c32) {};
\node [vertex] (v33) at (c33) {};
\node [vertex] (v34) at (c34) {};
\node [vertex] (v35) at (c35) {};
\node [vertex] (v36) at (c36) {};
\node [vertex] (v37) at (c37) {};
\node [vertex] (v38) at (c38) {};
\node [vertex] (v39) at (c39) {};
\node [vertex] (v40) at (c40) {};
\node [vertex] (v41) at (c41) {};
\node [vertex] (v42) at (c42) {};
\node [vertex] (v43) at (c43) {};
\node [vertex] (v44) at (c44) {};
\node [vertex] (v45) at (c45) {};
\node [vertex] (v46) at (c46) {};
\node [vertex] (v47) at (c47) {};
\node [vertex] (v48) at (c48) {};
\node [vertex] (v49) at (c49) {};
\node [vertex] (v50) at (c50) {};
\node [vertex] (v51) at (c51) {};
\node [vertex] (v52) at (c52) {};
\node [vertex] (v53) at (c53) {};
\node [vertex] (v54) at (c54) {};
\node [vertex] (v55) at (c55) {};
\draw (v0) -- (v1);
\draw (v0) -- (v2);
\draw (v0) -- (v3);
\draw (v1) -- (v4);
\draw (v1) -- (v5);
\draw (v2) -- (v6);
\draw (v2) -- (v7);
\draw (v3) -- (v8);
\draw (v3) -- (v9);
\draw (v4) -- (v10);
\draw (v4) -- (v11);
\draw (v5) -- (v12);
\draw (v5) -- (v13);
\draw (v6) -- (v13);
\draw (v6) -- (v14);
\draw (v7) -- (v15);
\draw (v7) -- (v16);
\draw (v8) -- (v16);
\draw (v8) -- (v17);
\draw (v9) -- (v10);
\draw (v9) -- (v18);
\draw (v10) -- (v19);
\draw (v11) -- (v20);
\draw (v11) -- (v21);
\draw (v12) -- (v21);
\draw (v12) -- (v22);
\draw (v13) -- (v23);
\draw (v14) -- (v15);
\draw (v14) -- (v24);
\draw (v15) -- (v25);
\draw (v16) -- (v26);
\draw (v17) -- (v27);
\draw (v17) -- (v28);
\draw (v18) -- (v28);
\draw (v18) -- (v29);
\draw (v19) -- (v20);
\draw (v19) -- (v29);
\draw (v20) -- (v30);
\draw (v21) -- (v31);
\draw (v22) -- (v32);
\draw (v22) -- (v33);
\draw (v23) -- (v24);
\draw (v23) -- (v33);
\draw (v24) -- (v34);
\draw (v25) -- (v26);
\draw (v25) -- (v34);
\draw (v26) -- (v35);
\draw (v27) -- (v35);
\draw (v27) -- (v36);
\draw (v28) -- (v37);
\draw (v29) -- (v38);
\draw (v30) -- (v31);
\draw (v30) -- (v38);
\draw (v31) -- (v39);
\draw (v32) -- (v39);
\draw (v32) -- (v40);
\draw (v33) -- (v41);
\draw (v34) -- (v42);
\draw (v35) -- (v43);
\draw (v36) -- (v44);
\draw (v36) -- (v45);
\draw (v37) -- (v45);
\draw (v37) -- (v46);
\draw (v38) -- (v46);
\draw (v39) -- (v47);
\draw (v40) -- (v48);
\draw (v40) -- (v49);
\draw (v41) -- (v42);
\draw (v41) -- (v49);
\draw (v42) -- (v43);
\draw (v43) -- (v50);
\draw (v44) -- (v50);
\draw (v44) -- (v51);
\draw (v45) -- (v52);
\draw (v46) -- (v47);
\draw (v47) -- (v53);
\draw (v48) -- (v53);
\draw (v48) -- (v54);
\draw (v49) -- (v55);
\draw (v50) -- (v55);
\draw (v51) -- (v52);
\draw (v51) -- (v54);
\draw (v52) -- (v53);
\draw (v54) -- (v55);
\end{tikzpicture}
\caption{A $56$-vertex fullerene with $1746$ Hamilton cycles} 
\label{fig:full56}
\end{center}
\end{figure}
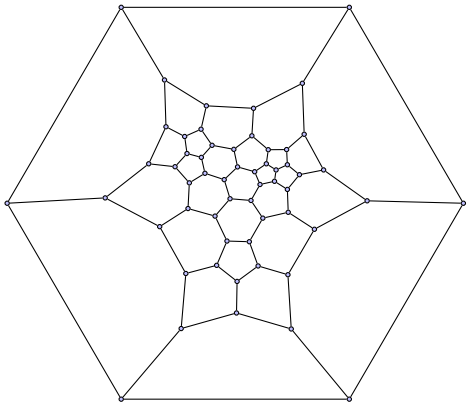

\bibliography{../gordonmaster}

\begin{thebibliography}{1}

\bibitem{MR2357364}
Gunnar Brinkmann and Brendan~D. McKay.
\newblock Fast generation of planar graphs.
\newblock {\em MATCH Commun. Math. Comput. Chem.}, 58(2):323--357, 2007.

\bibitem{MR2951794}
Gek~L. Chia and Carsten Thomassen.
\newblock On the number of longest and almost longest cycles in cubic graphs.
\newblock {\em Ars Combin.}, 104:307--320, 2012.

\bibitem{MR2698714}
Thomas~George Fowler.
\newblock {\em Unique coloring of planar graphs}.
\newblock ProQuest LLC, Ann Arbor, MI, 1998.
\newblock Thesis (Ph.D.)--Georgia Institute of Technology.

\bibitem{1812.05650}
Jan Goedgebeur, Barbara Meersman, and Carol~T. Zamfirescu.
\newblock Graphs with few hamiltonian cycles.
\newblock \url{https://arxiv.org/abs/1812.05650}, 2018.

\bibitem{MR1007713}
Allen~J. Schwenk.
\newblock Enumeration of {H}amiltonian cycles in certain generalized {P}etersen
  graphs.
\newblock {\em J. Combin. Theory Ser. B}, 47(1):53--59, 1989.

\bibitem{OEIS}
N.~J.~A. Sloane.
\newblock The {O}n-{L}ine {E}ncylopaedia of {I}nteger {S}equences.

\bibitem{MR0019300}
W.~T. Tutte.
\newblock On {H}amiltonian circuits.
\newblock {\em J. London Math. Soc.}, 21:98--101, 1946.

\end{thebibliography}
\bibliographystyle{plain}

\end{document}